\documentclass[10pt]{amsart}

\usepackage[utf8]{inputenc}
\usepackage[T1]{fontenc}
\usepackage[backend=biber,
            isbn=false,
            doi=false,
            maxbibnames=5,
            giveninits=true,
            style=alphabetic,
            citestyle=alphabetic]{biblatex}
\usepackage{url}
\renewbibmacro{in:}{}
\bibliography{paper.bib}

\usepackage{xcolor}
\definecolor{e-mail}{rgb}{0,.40,.80}
\definecolor{reference}{rgb}{.20,.60,.22}
\definecolor{citation}{rgb}{0,.40,.80}

\usepackage[colorlinks=true,
            linkcolor=reference,
            citecolor=citation,
            urlcolor=e-mail]{hyperref}
\usepackage[all, 2cell]{xy}
\UseTwocells
\usepackage{mathtools}
\usepackage{tikz-cd}
\usepackage{amssymb}
\usepackage{eucal}
\usepackage{cleveref}
\usepackage{fullpage}

\newtheorem{thm}{Theorem}[section]
\crefname{thm}{theorem}{theorems}
\newtheorem*{theorem}{Theorem}
\newtheorem{prop}[thm]{Proposition}
\crefname{prop}{proposition}{propositions}
\newtheorem*{proposition}{Proposition}
\newtheorem{cor}[thm]{Corollary}
\newtheorem{lm}[thm]{Lemma}
\theoremstyle{definition}
\newtheorem{defn}[thm]{Definition}
\theoremstyle{remark}
\newtheorem{remark}[thm]{Remark}
\newtheorem{example}[thm]{Example}

\newcommand{\cA}{\mathcal{A}}
\newcommand{\cB}{\mathcal{B}}
\newcommand{\cC}{\mathcal{C}}
\newcommand{\cD}{\mathcal{D}}
\newcommand{\cE}{\mathcal{E}}
\newcommand{\cF}{\mathcal{F}}

\newcommand{\cL}{\mathcal{L}}
\newcommand{\cM}{\mathcal{M}}

\newcommand{\cO}{\mathcal{O}}

\newcommand{\B}{\mathrm{B}}
\newcommand{\C}{\mathbf{C}}
\newcommand{\D}{\mathrm{D}}
\renewcommand{\L}{\mathrm{L}}
\newcommand{\N}{\mathrm{N}}
\newcommand{\R}{\mathrm{R}}
\newcommand{\T}{\mathrm{T}}
\newcommand{\U}{\mathrm{U}}
\newcommand{\Z}{\mathbf{Z}}
\newcommand{\rZ}{\mathrm{Z}}

\newcommand{\bZ}{\mathbb{Z}}

\renewcommand{\b}{\mathfrak{b}}
\newcommand{\g}{\mathfrak{g}}
\newcommand{\h}{\mathfrak{h}}
\newcommand{\n}{\mathfrak{n}}
\renewcommand{\sl}{\mathfrak{sl}}

\newcommand{\act}{\mathrm{act}}
\DeclareMathOperator{\ad}{ad}

\newcommand{\bimod}{\mathrm{bimod}}
\newcommand{\BiMod}[2]{{}_{#1}\mathrm{BMod}_{#2}}
\newcommand{\botimes}{\overline{\otimes}}
\newcommand{\Bun}{\mathrm{Bun}}
\newcommand{\bu}{\mathbf{1}}
\newcommand{\Cat}{\mathrm{Cat}}

\newcommand{\coev}{\mathrm{coev}}
\newcommand{\coind}{\mathrm{coind}}
\newcommand{\cp}{\mathrm{cp}}
\newcommand{\CoMod}{\mathrm{CoMod}}
\newcommand{\Dist}{\mathrm{Dist}}

\newcommand{\Day}{\mathrm{Day}}

\newcommand{\End}{\mathrm{End}}
\newcommand{\ev}{\mathrm{ev}}
\newcommand{\fd}{\mathrm{fd}}
\newcommand{\forget}{\mathrm{forget}}
\newcommand{\Frac}{\mathrm{Frac}}

\newcommand{\free}{\mathrm{free}}
\newcommand{\Fun}{\mathrm{Fun}}
\newcommand{\gen}{\mathrm{gen}}
\newcommand{\Gm}{\mathbb{G}_{\mathrm{m}}}
\newcommand{\HC}{\mathrm{HC}}
\newcommand{\hc}{\mathrm{hc}}
\newcommand{\Hom}{\mathrm{Hom}}

\newcommand{\id}{\mathrm{id}}
\newcommand{\Ind}{\mathrm{Ind}}
\newcommand{\ind}{\mathrm{ind}}
\newcommand{\Irr}{\mathrm{Irr}}
\newcommand{\lf}{\mathrm{lf}}
\newcommand{\LMod}{\mathrm{LMod}}

\newcommand{\op}{\mathrm{op}}
\newcommand{\mop}{\otimes\op}

\newcommand{\PrL}{\mathrm{Pr}^{\mathrm{L}}}
\newcommand{\bq}{\breve{q}}
\newcommand{\QCoh}{\mathrm{QCoh}}
\DeclareMathOperator{\Rep}{Rep}
\newcommand{\res}{\mathrm{res}}
\newcommand{\SL}{\mathrm{SL}}
\newcommand{\St}{\mathrm{St}}
\newcommand{\Sym}{\mathrm{Sym}}

\newcommand{\Uq}{\mathrm{U}_q}
\newcommand{\univ}{\mathrm{univ}}
\renewcommand{\Vec}{\mathrm{Vec}}
\newcommand{\Vect}{\mathrm{Vect}}
\newcommand{\wt}{\mathrm{wt}}
\newcommand{\ZDr}{\mathrm{Z}_{\mathrm{Dr}}}

\newcommand{\ham}{/\!\!/}
\newcommand{\defterm}[1]{\textbf{\emph{#1}}}
\newcommand{\adj}[2]{
\xymatrix{
#1 \ar@<.5ex>[r] & #2 \ar@<.5ex>[l]
}
}

\begin{document}
\title{A categorical approach to dynamical quantum groups}
\address{Institut f\"{u}r Mathematik, Universit\"{a}t Z\"{u}rich, Zurich, Switzerland}
\email{artem.kalmykov@math.uzh.ch}
\author{Artem Kalmykov}
\address{Institut f\"{u}r Mathematik, Universit\"{a}t Z\"{u}rich, Zurich, Switzerland}
\email{pavel.safronov@math.uzh.ch}
\author{Pavel Safronov}
\begin{abstract}
We present a categorical point of view on dynamical quantum groups in terms of categories of Harish-Chandra bimodules. We prove Tannaka duality theorems for forgetful functors into the monoidal category of Harish-Chandra bimodules in terms of a slight modification of the notion of a bialgebroid. Moreover, we show that the standard dynamical quantum groups $F(G)$ and $F_q(G)$ are related to parabolic restriction functors for classical and quantum Harish-Chandra bimodules. Finally, we exhibit a natural Weyl symmetry of the parabolic restriction functor using Zhelobenko operators and show that it gives rise to the action of the dynamical Weyl group.
\end{abstract}
\maketitle

\section*{Introduction}

\subsection*{Categorical approach to quantum groups}

Let $G$ be an affine algebraic group over a field $k$. The Tannaka duality theorems \cite{Saavedra,Deligne} imply that one can uniquely reconstruct $G$ from the data of a symmetric monoidal category $\Rep(G)$ of $G$-representations and the forgetful symmetric monoidal functor
\begin{equation}
F\colon \Rep(G)\longrightarrow \Vect.
\end{equation}
Namely, $F$ admits a right adjoint $F^\R\colon \Vect\rightarrow \Rep(G)$ and the algebra $\cO(G)$ of polynomial functions on $G$ can be reconstructed as
\[\cO(G)\cong FF^\R(k),\]
where the Hopf algebra structure on $\cO(G)$ is reconstructed from the monoidal structure on $F$.

Suppose $G$ is a reductive algebraic group, $q\in\C^\times$ and consider the category $\Rep_q(G)$ of representations of the quantum group with divided powers \cite{Lusztig,ChariPressley}. Then $\Rep_q(G)$ carries a natural braided monoidal structure and the forgetful functor
\begin{equation}
F\colon \Rep_q(G)\longrightarrow \Vect
\label{eq:RepqGforgetful}
\end{equation}
is merely monoidal. In the same way the Hopf algebra $\cO_q(G)$ of functions on the quantum group is reconstructed as $FF^\R(k)$.

The failure of the forgetful functor to preserve the braiding is captured by the \emph{$R$-matrix} (see \cref{def:Rmatrix}), i.e. a collection of maps
\[R_{V, W}\colon V\otimes W\longrightarrow V\otimes W\]
for two representations $V,W\in\Rep_q(G)$. Moreover, for three representations $U,V,W\in\Rep_q(G)$ the $R$-matrix satisfies the \emph{Yang--Baxter equation}
\begin{equation}
R_{UV} R_{UW} R_{VW} = R_{VW} R_{UW} R_{UV}
\end{equation}
in $\End(U\otimes V\otimes W)$.

\subsection*{Dynamical quantum groups}

In several areas of mathematical physics a version of the above equation has appeared for a \emph{dynamical $R$-matrix} $R_{V, W}(\lambda)\colon V\otimes W\rightarrow V\otimes W$ which depends on a parameter $\lambda\in\h^*$ (dual space of the Cartan subalgebra $\h\subset \g$); the corresponding \emph{dynamical Yang--Baxter equation} is
\begin{equation}
R_{UV}(\lambda- h^{(3)}) R_{UW}(\lambda) R_{VW}(\lambda-h^{(1)}) = R_{VW}(\lambda) R_{UW}(\lambda-h^{(2)}) R_{UV}(\lambda),
\label{eq:DYBE}
\end{equation}
where the shifts refer to the $\h$-weights of the corresponding elements of $U\otimes V\otimes W$. We refer to $\U\h\cong \cO(\h^*)$ as the base of the dynamical quantum group. As explained by Felder \cite{Felder}, the equation \eqref{eq:DYBE} is closely related to the star-triangle relation for face-type statistical mechanical models \cite{Baxter}. Moreover, it naturally appears in the description of the exchange algebra in Liouville and Toda conformal field theories \cite{GervaisNeveu}. The study of the dynamical $R$-matrix gave rise to the theory of dynamical quantum groups; see \cite{EtingofSchiffmann,EtingofICM} for reviews.

On the classical level ordinary quantum groups correspond to Poisson-Lie structures on $G$ \cite{DrinfeldICM}. Similarly, dynamical quantum groups correspond to dynamical Poisson groupoid structures on the trivial groupoid $\h^*\times G\times \h^*\rightrightarrows \h^*$ (see \cite{LGSX} for Poisson groupoids and \cite{EtingofVarchenkoGeometry} for the dynamical version). After quantization ordinary quantum groups become Hopf algebras while dynamical quantum groups become bialgebroids or Hopf algebroids (see \cite{Takeuchi} for the original definition of bialgebroids, \cite{Lu,Xu} for Hopf algebroids and \cite{EtingofVarchenkoSolutions} for the dynamical version).

One is naturally led to wonder about the categorical interpretation of dynamical quantum groups similar to the categorical interpretation \eqref{eq:RepqGforgetful} of ordinary quantum groups explained above. Our first goal is to develop such an approach (inspired by a previous work by Donin and Mudrov \cite{DoninMudrovVectorBundles,DoninMudrovQuantumGroupoids}) and prove Tannaka-type reconstruction statements.

\subsection*{Dynamical quantum groups via Harish-Chandra bimodules}

An important object in representation theory is the category $\HC(G)$ of \emph{Harish-Chandra bimodules}: the monoidal category of $\U\g$-bimodules with an integrable diagonal action. As we will explain shortly, the theory of dynamical quantum groups turns out to be closely related to the category $\HC(H)$ of Harish-Chandra bimodules for a torus $H$. In the main body of the paper (see \cref{sect:HCgeneral}), we present a general formalism which incorporates classical and quantum examples as well as non-abelian bases (following \cite{SafronovQMM}), but for simplicity here we stick to the case of $\HC(H)$.

First, we introduce the notion of a Harish-Chandra bialgebroid, which is a slight variant of the notion of an $\h$-bialgebroid introduced in \cite[Section 4.1]{EtingofVarchenkoSolutions}, see \cref{def:HCbialgebroid} for the general definition and \cref{ex:hbiaglebroids} for the case of $\HC(H)$. Namely, it is a bigraded algebra $B=\oplus_{\alpha, \beta\in\Lambda} B_{\alpha\beta}$, where $\Lambda$ is the character lattice of $H$, together with two quantum moment maps $s,t\colon \cO(\h^*)\rightarrow B$, a coproduct $\Delta\colon B\rightarrow B\times_{\U\h} B$, where the Takeuchi product introduced in \cite{Takeuchi} is
\[(B\times_{\U\h} B)_{\alpha\beta} = \bigoplus_{\delta\in\Lambda} B_{\alpha\delta}\otimes_{\cO(\h^*)} B_{\delta\beta},\]
and a counit $\epsilon\colon B\rightarrow \D(H)$ into the algebra of differential operators on $H$. We prove the following equivalent characterization of Harish-Chandra bialgebroids (see \cref{thm:HClaxmonoidalcomonad}).

\begin{theorem}
A colimit-preserving lax monoidal comonad $\bot\colon \HC(H)\rightarrow \HC(H)$ is the same as a Harish-Chandra bialgebroid $B$, so that $\bot(M) = B\times_{\U\h} M$.
\end{theorem}

We may similarly define comodules over a Harish-Chandra bialgebroid in terms of a $\Lambda$-graded $\cO(\h^*)$-module $M=\oplus_{\alpha\in\Lambda} M_\alpha$ together with a coaction map $M\rightarrow B\times_{\U\h} M$. We prove the following Tannaka reconstruction theorem (see \cref{HarishChandraTannakaReconstruction}).

\begin{theorem}
Suppose $\cD$ is a monoidal category with a monoidal functor $F\colon \cD\rightarrow \HC(H)$ which admits a colimit-preserving right adjoint $F^\R\colon \HC(H)\rightarrow \cD$. Then there is a Harish-Chandra bialgebroid $B$, such that $(F\circ F^\R)(-)\cong B\times_{\U\h}(-)$ and $F$ factors through a monoidal functor
\[\cD\longrightarrow \CoMod_B(\HC(H)).\]
If $F$ is conservative and preserves equalizers, the above functor is an equivalence.
\end{theorem}

Let us now explain the origin of dynamical R-matrices. Assume that $\cD$ in addition has a braided monoidal structure. Moreover, assume that the functor $F\colon \cD\rightarrow \HC(H)$ lands in free Harish-Chandra bimodules, i.e. there is a functor $F'\colon \cD\rightarrow \Rep(H)$ and an equivalence $F(x)\cong \U\h\otimes F'(x)$ for any object $x\in\cD$. The following is \cref{prop:dynamicalFRT}.

\begin{proposition}
Under the above assumptions the image of the braiding under $F\colon \cD\rightarrow \HC(H)$ gives rise to dynamical $R$-matrices $R\colon \h^*\rightarrow \End(F'(x)\otimes F'(y))$ satisfying the dynamical Yang--Baxter equation \eqref{eq:DYBE}.
\end{proposition}

The above proposition is a direct quantum analogue of an interpretation of classical dynamical $r$-matrices in terms of 1-shifted Poisson morphisms (see \cite{CPTVV} for what this means) $[\h^*/H]\rightarrow \B G$, see \cite[Proposition 5.7]{SafronovPoissonLie}.

Let us compare these results to Tannaka reconstruction results for bialgebroids proven in \cite{Szlachanyi,ShimizuTannaka}. Suppose $R$ is a ring. It is shown in \cite[Theorem 5.4]{Szlachanyi} that a colimit-preserving \emph{oplax monoidal monad} on the category $\BiMod{R}{R}$ of $R$-bimodules is the same as a bialgebroid over $R$. Comparing it to our \cref{thm:HClaxmonoidalcomonad}, the difference is that we work with \emph{lax monoidal comonads} instead, replace $\BiMod{\U\h}{\U\h}$ by the full subcategory $\HC(H)$ of Harish-Chandra bimodules and replace Takeuchi's bialgebroids by $\h$-bialgebroids (i.e. adding an extra integrability assumption).

Szlach\'{a}nyi \cite[Theorem 3.6]{Szlachanyi} has proven a Tannaka-type reconstruction result for monoidal functors $F\colon \cD\rightarrow \BiMod{R}{R}$ admitting \emph{left adjoints} in terms of \emph{modules} over the corresponding bialgebroid. Shimizu has also proven a version of such a Tannaka reconstruction result in terms of \emph{comodules} over the bialgebroid (see \cite[Theorem 4.3, Lemma 4.18]{ShimizuTannaka}).

\subsection*{Parabolic restriction}

There are two standard dynamical quantum groups $F(G)$ and $F_q(G)$ introduced in \cite{EtingofVarchenkoExchange} in terms of the so-called \emph{exchange construction}. Here $F(G)$ quantizes the standard rational dynamical $r$-matrix and $F_q(G)$ quantizes the standard trigonometric dynamical $r$-matrix (see \cite[Section 4]{EtingofSchiffmann}). Our second goal of the paper is to relate these dynamical quantum groups to objects in geometric representation theory.

Let $G$ be a split reductive algebraic group over a characteristic zero field $k$, $B\subset G$ a Borel subgroup and $H=B/[B, B]$ the abstract Cartan subgroup; we denote by $\g,\b,\h$ their Lie algebras. Consider the correspondence of algebraic stacks
\begin{equation}
\xymatrix{
& [\b/B] \ar[dl] \ar[dr] & \\
[\g^*/G] && [\h^*/H].
}
\label{eq:correspondence}
\end{equation}
It appears in many areas of symplectic geometry and geometric representation theory:
\begin{itemize}
\item Let $\tilde{\g}$ be the variety parametrizing Borel subgroups of $G$ together with an element $x\in\g$ contained in the Lie algebra of the corresponding Borel subgroup. The projection $\tilde{\g}\rightarrow \g$ is known as the Grothendieck--Springer resolution (see \cite[Section 3.1.31]{ChrissGinzburg}). We may identify $[\tilde{\g}/G]\cong [\b/B]$, so that the projection $[\b/B]\rightarrow [\g^*/G]$ is identified with the Grothendieck--Springer resolution $[\tilde{\g}/G]\rightarrow [\g/G]$. The study of the categories of $\D$-modules on this correspondence is closely related to Springer theory (see \cite{Gunningham} and references there).

\item Let $N\subset B$ be the unipotent radical. Then we may identify
\[[\b/B]\cong [G\backslash \T^*(G/N) / H],\]
where $\T^*(G/N)/H\rightarrow G/B$ is the universal family of twisted cotangent bundles over the flag variety parametrized by $\lambda\in\h^*$. In particular, quantization of this correspondence is closely related to the Beilinson--Bernstein localization theorem \cite{BeilinsonBernstein} (see \cite{BZNadlerBB}).

\item The stacks $[\g^*/G],[\h^*/H]$ have 1-shifted symplectic structures in the sense of \cite{PTVV}; moreover, \eqref{eq:correspondence} is a 1-shifted Lagrangian correspondence. It is shown in \cite[Section 2.2.1]{CalaqueTFT} that a Lagrangian $L$ in $[\g^*/G]$ is the same as a Hamiltonian $G$-space, i.e. an algebraic symplectic variety $X$ equipped with a symplectic $G$-action and a moment map $X\rightarrow \g^*$. Composing the Lagrangian $L\rightarrow [\g^*/G]$ with the correspondence \eqref{eq:correspondence} we obtain a Lagrangian in $[\h^*/H]$, i.e. a Hamiltonian $H$-space. It is shown in \cite{SafronovImplosion} that this procedure coincides with the procedure of symplectic implosion \cite{GuilleminJeffreySjamaar,DancerKirwanSwann}.

\item One may replace Lie algebras by the corresponding groups, i.e. one may consider the correspondence $[G/G]\leftarrow [B/B]\rightarrow [H/H]$. It is shown in \cite[Theorem A]{Boalch} that this correspondence (and its analogue for a parabolic subgroup) appears in the description of logarithmic connections on a disk.
\end{itemize}

Consider the induced bimodule category
\begin{equation}
\QCoh([\g^*/G])\curvearrowright \QCoh([\b/B]) \curvearrowleft \QCoh([\h^*/H]),
\label{eq:classicalbimodule}
\end{equation}
where $\QCoh([\g^*/G])$ is the symmetric monoidal category of quasi-coherent sheaves on the stack $[\g^*/G]$. Explicitly, it can be identified as
\[\QCoh([\g^*/G])\cong \LMod_{\Sym(\g)}(\Rep G)\]
and similarly for $H$.

In \cref{sect:classicalHC} we study a quantum version of the bimodule \eqref{eq:classicalbimodule}:
\begin{equation}
\HC(G)\curvearrowright \cO^{\univ} \curvearrowleft \HC(H).
\label{eq:quantumbimodule}
\end{equation}
Here, as before, $\HC(G)$ is the monoidal category of Harish-Chandra bimodules, i.e. $\U\g$-bimodules with an integrable diagonal action. $\cO^{\univ}$ is a universal version of category $\cO$: it is the category of $\U\g$-modules internal to the category $\Rep(H)$ whose $\n$-action is locally nilpotent. Equivalently, it is the category of $(\U\g, \U\h)$-bimodules whose diagonal $B$-action is integrable. The module structure on either side is given by the tensor product of bimodules using the latter description of $\cO^{\univ}$. The universal Verma module $M^{\univ} = \U\g\otimes_{\U\b}\U\h$ is naturally an object of $\cO^{\univ}$.

Let us explain how it relates to the classical picture. The algebra $\U\g$ has a natural PBW filtration; consider the corresponding Rees algebra over $k[\hbar]$. The above constructions can be repeated to produce $k[\hbar]$-linear categories, so that at $\hbar=0$ the bimodule \eqref{eq:quantumbimodule} reduces to the bimodule \eqref{eq:classicalbimodule}.

Passing to the right adjoint of the action functor $\HC(H)\rightarrow \cO^{\univ}$ on the universal Verma module $M^{\univ}$, one obtains the parabolic restriction functor
\[\res\colon \HC(G)\longrightarrow \HC(H)\]
given by $\res(X) = (X/X\n)^N$, which is naturally lax monoidal. The following statement combines \cref{prop:HCalgebras} and \cref{rmk:Mickelsson} and provides a quantization of symplectic implosion.

\begin{proposition}
An algebra in $\HC(G)$ is a $G$-equivariant algebra $A$ with a quantum moment map $\U\g\rightarrow A$. We have an isomorphism of algebras $\res(A)\cong A\ham N$, where $A\ham N$ is the quantum Hamiltonian reduction by $N$.
\end{proposition}

For a generic central character $\chi\colon \rZ(\U\g)\rightarrow \C$, the BGG category $\cO_\chi$ with that central character is semisimple with simple objects given by Verma modules. We prove an analogous statement in the universal case. The following statement combines \cref{thm:Ofreerank1} and \cref{cor:resmonoidal}.

\begin{theorem}
Consider the subcategories $\HC(H)^{\gen}\subset \HC(H)$ and $\cO^{\univ, \gen}\subset \cO^{\univ}$ of modules with generic $\h$-weights. Then the functor $\HC(H)\rightarrow \cO^{\univ, \gen}$ is an equivalence. In particular,
\[\res^{\gen}\colon \HC(G)\longrightarrow \HC(H)^{\gen}\]
is strongly monoidal and colimit-preserving.
\end{theorem}

The key step in the above statement is to prove that the Verma module for generic highest weights is projective; in the universal setting this is captured by the existence of the \emph{extremal projector} \cite{AsherovaSmirnovTolstoy} (see \cref{thm:extremalprojector}) which splits the projection $\U\g\rightarrow M^{\univ}$ for generic $\h$-weights.

There is a natural monoidal functor $\free\colon\Rep(G)\rightarrow \HC(G)$ given by $V\mapsto \U\g\otimes V$, so we get a monoidal functor
\[\Rep(G)\xrightarrow{\free} \HC(G)\xrightarrow{\res^{\gen}} \HC(H)^{\gen}.\]

We, moreover, show in \cref{thm:FGexchange} that the $\h$-bialgebroid reconstructed from $\Rep(G) \rightarrow \HC(H)^{\gen}$ is isomorphic to $F(G)$, so that $\Rep(G)$ is equivalent to $F(G)$-comodules. We also prove analogous statements in the setting of quantum groups in \cref{sect:quantumres}.

These results have the following interpretation. The same braided monoidal category $\Rep_q(G)$ has different monoidal functors $\Rep_q(G)\rightarrow \Vect$ corresponding to different choices of the classical $r$-matrix; by Tannaka duality this corresponds to non-standard quantum groups, such as the Cremmer--Gervais quantum group in the case $G=\SL_n$. In this paper we study monoidal functors $\Rep_q(G)\rightarrow \HC_q(H)^{\gen}$ which give rise to dynamical quantum groups. Note that these are different ways to study the same braided monoidal category.

We also expect that the approach to dynamical quantum groups $F(G)$ and $F_q(G)$ presented here in terms of the correspondence \eqref{eq:correspondence} might be useful to have an interpretation of Felder's dynamical quantum group \cite{Felder} in terms of the 1-shifted Lagrangian correspondence $\Bun_G(E)\leftarrow \Bun_B(E)\rightarrow \Bun_H(E)$ of moduli stacks of bundles on an elliptic curve $E$. It is interesting to note that the same correspondence is closely related to Feigin--Odesskii algebras \cite{FeiginOdesskii} (in particular, Sklyanin algebras \cite{Sklyanin}), see \cite[Example 4.11]{SafronovPoissonLie} and \cite{HuaPolishchuk}.

It is shown in \cite[Theorem 3.11]{BZBJ2} that $\HC_q(G)$-module categories are the same as $\Rep_q(G)$-braided module categories \cite[Section 5.1]{Brochier}. In particular, the monoidal functor $\res^{\gen}\colon \HC_q(G)\rightarrow \HC_q(H)^{\gen}$ allows one to transfer $\Rep_q(H)$-braided module categories to $\Rep_q(G)$-braided module categories.

\subsection*{Dynamical Weyl group}

Let $W=\N(H)/H$ be the Weyl group and $\hat{W}$ the braid group covering $W$. The group $W$ naturally acts on the symmetric monoidal category $\Rep(H)$, so that we may consider the category of $W$-invariants $\Rep(H)^W$. Moreover, there exists a map $\hat{W}\rightarrow \N(H)$ lifting $\hat{W}\rightarrow W$ \cite{Tits}, so that the forgetful functor $\Rep(G)\rightarrow \Rep(H)$ factors through a symmetric monoidal functor
\begin{equation}
\Rep(G)\longrightarrow \Rep(H)^{\hat{W}}.
\end{equation}
Our third goal of the paper is to exhibit Weyl symmetry of the parabolic restriction functor for Harish-Chandra bimodules. A similar setup works for quantum groups using the quantum Weyl group \cite{Lusztig,Soibelman,KirillovReshetikhin}. Note, however, that the resulting functor
\begin{equation}
\Rep_q(G)\longrightarrow \Rep_q(H)^{\hat{W}}
\end{equation}
is \emph{not} monoidal: in fact, the failure of the quantum Weyl group to be monoidal is related to the failure of the functor $\Rep_q(G)\rightarrow \Rep_q(H)$ to be braided; this can be encapsulated in the notion of a braided Coxeter category \cite{AppelToledanoLaredo}.

Zhelobenko \cite{ZhelobenkoCocycles} in the study of Mickelsson algebras has introduced a collection of \emph{Zhelobenko operators} $q_w\colon \U\g\rightarrow \U\g$ for every element of the Weyl group $w\in W$ satisfying the braid relations (see \cref{thm:classicalZhelobenko}). It was realized in \cite{KhoroshkinOgievetsky} that these operators give an action of the braid group $\hat{W}$ on a localized Mickelsson algebra.

Consider the $W$-action on $\HC(H)$, where $W$ acts on $\U\h$ via the \emph{dot action} (the usual $W$-action shifted by the half-sum of positive roots $\rho$) and on $H$ via the usual action. The above results directly imply the following statement (see \cref{thm:classicalZhelobenkoFactorization}).

\begin{theorem}
The Zhelobenko operators define a monoidal functor
\[\res^{\gen}\colon \HC(G)\longrightarrow \HC(H)^{\gen, \hat{W}}\]
lifting $\res^{\gen}\colon \HC(G)\rightarrow \HC(H)^{\gen}$.
\end{theorem}

Suppose $V\in\Rep(G)$. Then $\res^{\gen}(\U\g\otimes V)\cong (\U\h)^{\gen}\otimes V$, where $(\U\h)^{\gen}\supset \U\h$ is a certain localization (see \cref{def:genericweight}). In particular, the $\hat{W}$-symmetry is captured by certain rational maps $A_{w, V}\colon \h^*\rightarrow \End(V)$ satisfying the braid relation. We prove in \cref{thm:TVdynamicalWeylGroup} that these coincide with the \emph{dynamical Weyl group} operators introduced in \cite{TarasovVarchenko,EtingofVarchenkoDynamicalWeyl}.

Let us mention a relationship between these results and the \emph{generalized Harish--Chandra isomorphism} \cite{KhoroshkinNazarovVinberg}. Consider the functor $\widetilde{\res}\colon \HC(G)\rightarrow \HC(H)$ given by $\tilde{\res}(X) = \n_- X \backslash X / X\n_+$. There is a natural transformation $\res(X)\rightarrow \widetilde{\res}(X)$ which becomes an isomorphism in $\HC(H)^{\gen}$ (see \cref{prop:coinvinv}). We obtain a restriction map
\begin{equation}
\xymatrix{
\Hom_{\HC(G)}(\U\g, \U\g\otimes V) \ar^{\widetilde{\res}}[r] \ar^{\sim}[d] & \Hom_{\HC(H)}(\U\h, \U\h\otimes V) \ar^{\sim}[d] \\
(\U\g\otimes V)^G \ar[r] & \U\h\otimes V^H
}
\end{equation}
The object $(\U\h)^{\gen}\in\HC(H)^{\gen}$ has a canonical $\hat{W}$-equivariance structure given by the dot action of $W$ on $\U\h$. In particular, Zhelobenko operators define maps $\U\h\otimes V^H\rightarrow (\U\h)^{\gen}\otimes V^H$ and, in fact, the action factors through the action of the Weyl group. The resulting homomorphism
\[\widetilde{\res}\colon (\U\g\otimes V)^G\longrightarrow (\U\h\otimes V^H)^W\]
is shown in \cite{KhoroshkinNazarovVinberg} to be an isomorphism. It generalizes the usual Harish-Chandra isomorphism (see e.g. \cite[Theorem 1.10]{HumphreysBGG}) which is obtained for $V=k$ the trivial one-dimensional representation.

The papers \cite{BravermanFinkelberg,GinzburgRiche} gave an interpretation of the dynamical Weyl group in terms of equivariant cohomology of the affine Grassmannian of the Langlands dual group, using the geometric Satake equivalence. It would be interesting to see the appearance of the Zhelobenko operators using the Langlands dual interpretation of Harish-Chandra bimodules from \cite{BezrukavnikovFinkelberg}.

Let us mention a categorical point of view on the Weyl symmetry of the parabolic restriction functor $\res^{\gen}\colon \HC(G)\rightarrow \HC(H)$. By abstract reasons the action functor $\HC(G)\rightarrow \cO^{\univ}$ factors through the category of coalgebras over a comonad $\St\colon \cO^{\univ}\rightarrow \cO^{\univ}$ obtained from the right adjoint of the action functor. In particular, for generic weights parabolic restriction factors through the category of $\St$-coalgebras in $\HC(H)^{\gen}$. We expect that there is an equivalence between $\St$ and the comonad corresponding to the $W$-action on $\HC(H)^{\gen}$. We refer to \cite{BZNadlerBB}, where it is called the Weyl comonad, and \cite[Theorem 4.6]{Gunningham} for an analogous theorem in the setting of $\D$-modules.

\subsection*{Acknowledgements}

The authors were supported by the NCCR SwissMAP grant of the Swiss National Science Foundation. We would like to thank Adrien Brochier, Damien Calaque, Sam Gunningham and David Jordan for useful discussions.

\section{Background}

In this section we recall some facts about locally presentable categories, cp-rigid monoidal categories and Tannaka reconstruction for bialgebras.

\subsection{Locally presentable categories}

Let $k$ be a field. All categories and functors we will consider are $k$-linear. Throughout this paper we work with locally presentable categories (we refer to \cite{AdamekRosicky} and \cite[Section 2]{BCJF} for more details). Here are the main examples:
\begin{itemize}
\item If $\cC$ is a small category, the category of presheaves $\Fun(\cC^{\op}, \Vect)$ is locally presentable. For instance, this applies to the category $\LMod_A$ of (left) modules over a $k$-algebra $A$.

\item If $\cC$ is a small category which admits finite colimits, the ind-completion $\Ind(\cC)$ (see \cite[Chapter 6]{KashiwaraSchapira} for what it means) is locally presentable.

\item If $C$ is a $k$-coalgebra, the category of $C$-comodules $\CoMod_C$ is locally presentable (see \cite[Corollary 9]{Wischnewsky} noting that a Grothendieck category is locally presentable). In fact, $\CoMod_C$ is the ind-completion of the category of finite-dimensional $C$-comodules (see \cite[Corollaire 2.2.2.3]{Saavedra}).

\item If $\cC, \cD$ are locally presentable categories, the category $\Fun^\L(\cC, \cD)$ of colimit-preserving functors from $\cC$ to $\cD$ is locally presentable.
\end{itemize}

It turns out that many examples of locally presentable categories are, in fact, presheaf categories.

\begin{defn}
Let $\cC$ be a locally presentable category. An object $x\in\cC$ is \defterm{compact projective} if $\Hom_\cC(x, -)\colon \cC\rightarrow \Vect$ preserves colimits. $\cC$ has \defterm{enough compact projectives} if every object receives a nonzero morphism from a compact projective.
\end{defn}

We denote by $\cC^{\cp}\subset \cC$ the full subcategory of compact projective objects.

\begin{prop}
Suppose $\cC$ has enough compact projectives. Then the functor
\[\cC\longrightarrow \Fun((\cC^{\cp})^\op, \Vect)\]
given by $x\mapsto (y\mapsto \Hom_\cC(y, x))$ is an equivalence.
\end{prop}

Locally presentable categories naturally form a symmetric monoidal 2-category $\PrL$ \cite{Bird}:
\begin{itemize}
\item Its objects are locally presentable categories.
\item Its 1-morphisms are colimit-preserving functors.
\item Its 2-morphisms are natural transformations.
\item The tensor product is uniquely determined by the following property: for $\cC,\cD,\cE\in\PrL$ a colimit-preserving functor $\cC\otimes \cD\rightarrow \cE$ is the same as a bifunctor $\cC\times\cD\rightarrow \cE$ preserving colimits in each variables.
\item The unit is $\Vect\in\PrL$.
\end{itemize}

An important fact about locally presentable categories is that a colimit-preserving functor between locally presentable categories admits a right adjoint. We will now write a formula for the adjoint assuming the source category has enough compact projectives. Let us first recall the notion of a coend (see \cite{Loregian} for more details on coends).

\begin{defn}
Suppose $\cC$ and $\cD$. The \defterm{coend} of a bifunctor $F\colon \cC\times \cC^{\op}\rightarrow \cD$ is the coequalizer
\[\int^{x\in\cC} F(x, x) =
\mathrm{coeq}\left(
\xymatrix{
\coprod_{x\rightarrow y} F(x, y) \ar@<.5ex>[r] \ar@<-.5ex>[r] & \coprod_{x} F(x, x)
}
\right).
\]
\end{defn}

We will use the following Yoneda-like property of coends (see \cite[Proposition 2.2.1]{Loregian}).

\begin{prop}
For any functors $F\colon \cC\rightarrow \cD, G\colon \cC^{\op}\rightarrow \cD$ we have natural isomorphisms
\[\int^{x\in\cC} \Hom_\cC(x, y)\otimes F(x)\cong F(y),\qquad \int^{x\in\cC} \Hom_\cC(y, x)\otimes G(x)\cong G(y).\]
\label{prop:ninjaYoneda}
\end{prop}

The following is an immediate corollary.

\begin{prop}
Suppose $F\colon \cC\rightarrow \cD$ is a colimit-preserving functor of locally presentable categories, where $\cC$ has enough compact projectives. Then the right adjoint is given by the coend
\[F^\R(x) = \int^{y\in\cC^{\cp}} \Hom_\cD(F(y), x)\otimes y.\]
The counit of the adjunction $FF^\R(x)\rightarrow x$ is given by the evaluation map $\Hom_\cD(F(y), x)\otimes F(y)\rightarrow x$; the unit of the adjunction $z\rightarrow F^\R F(z)$ is given by including the identity map $\id\colon F(z)\rightarrow F(z)$ in the coend.
\label{prop:rightadjoint}
\end{prop}

\subsection{Cp-rigidity}

By convention all monoidal categories $\cC$ we consider in this paper are locally presentable such that the tensor product bifunctor $\cC\times\cC\rightarrow \cC$ preserves colimits in each variable. So, by the universal property of the tensor product in $\PrL$ it descends to a colimit-preserving functor
\[T\colon \cC\otimes\cC\longrightarrow \cC.\]
We denote by $\cC^{\mop}$ the same category with the opposite monoidal structure.

We will consider rigid monoidal categories in the text. Since we work with large categories, we cannot expect all objects to be dualizable (as in the category of all vector spaces); instead, we will restrict our attention to compact projective objects.

\begin{defn}
Let $\cC$ be a monoidal category with enough compact projectives. It is \defterm{cp-rigid} if every compact projective object admits left and right duals.
\end{defn}

\begin{lm}
Suppose $\cC$ is a cp-rigid monoidal category and $x,y\in\cC$ are compact projective objects. Then $x\otimes y$ is also compact projective.
\label{lm:cptensorproduct}
\end{lm}
\begin{proof}
We have
\[
\Hom_\cC(x\otimes y, -) \cong \Hom_\cC(x, (-)\otimes y^\vee).
\]
By assumption the tensor product preserves colimits in each variable, so $(-)\otimes y^\vee$ is colimit-preserving. Since $x$ is compact projective, $\Hom_\cC(x, -)$ is colimit-preserving. Therefore, $\Hom_\cC(x\otimes y, -)$ is also colimit-preserving.
\end{proof}

If $\cC$ is cp-rigid, the tensor product functor $T\colon\cC\otimes \cC\rightarrow \cC$ admits a colimit-preserving right adjoint $T^\R\colon \cC\rightarrow \cC\otimes \cC$ (see e.g. \cite[Section 5.3]{BJS}). It has the following explicit formula.

\begin{prop}
Suppose $\cC$ is a cp-rigid monoidal category. Then
\[T^\R(y) \cong \int^{x\in\cC^{\cp}} (y\otimes x^\vee)\boxtimes x.\]
\label{prop:TRformula}
\end{prop}
\begin{proof}
By \cref{prop:rightadjoint} the right adjoint is
\[T^\R(y) \cong \int^{x_1,x_2\in\cC^{\cp}} \Hom_\cC(x_1\otimes x_2, y)\otimes (x_1\boxtimes x_2).\]
Since compact projective objects in $\cC$ are dualizable, we can rewrite it as
\begin{align*}
T^\R(y)&\cong \int^{x_1,x_2\in\cC^{\cp}} \Hom_\cC(x_1, y\otimes x_2^\vee)\otimes (x_1\boxtimes x_2) \\
&\cong \int^{x_2\in\cC^{\cp}} (y\otimes x_2^\vee)\boxtimes x_2,
\end{align*}
where in the last isomorphism we have used \cref{prop:ninjaYoneda}.
\end{proof}

Consider $\cC\otimes \cC$ as a $\cC\otimes \cC^{\mop}$-module category via the left action on the first factor and the right action on the second factor. By \cite[Proposition 4.1]{BJS} $T^\R$ is a functor of $\cC\otimes\cC^{\mop}$-module categories. This can be expressed in the following isomorphism.

\begin{prop}
Suppose $\cC$ is as before. Then $T^\R\colon \cC\rightarrow \cC\otimes \cC$ is a functor of $\cC\otimes\cC^{\mop}$-module categories. Concretely, for any object $y\in\cC$ there is a natural isomorphism
\[\int^{x\in\cC^{\cp}} (y\otimes x^\vee)\boxtimes x\cong \int^{x\in\cC^{\cp}} x^\vee\boxtimes (x\otimes y)\]
which is given for a compact projective $y\in\cC$ by
\[x^\vee\boxtimes(x\otimes y)\xrightarrow{\coev_y\otimes \id} (y\otimes y^\vee\otimes x^\vee)\boxtimes (x\otimes y)\xrightarrow{\pi_{x\otimes y}} \int^{x\in\cC^{\cp}} (y\otimes x^\vee)\boxtimes x.\]
\label{prop:TRmodulefunctor}
\end{prop}

\begin{cor}
The object $T^\R(\bu)\in \cC\otimes \cC^{\mop}$ has a natural algebra structure.
\label{cor:TRalgebra}
\end{cor}
\begin{proof}
$T^\R T$ is naturally a monad on $\cC\otimes \cC^{\mop}$. By definition $T\colon \cC\otimes \cC\rightarrow \cC$ is a functor $\cC\otimes \cC^{\mop}$-module categories. By \cref{prop:TRmodulefunctor} $T^\R\colon \cC\rightarrow \cC\otimes \cC$ is also a functor of $\cC\otimes \cC^{\mop}$-module categories. Therefore, $(T^\R T)(\bu_{\cC\otimes \cC})$ has a natural algebra structure.
\end{proof}

The key property of cp-rigid monoidal categories is that they are canonically self-dual objects of $\PrL$.

\begin{thm}
Let $\cC$ be a cp-rigid monoidal category with a compact projective unit. The evaluation and coevaluation pairings
\begin{align}
\ev&\colon \cC\otimes\cC\xrightarrow{T}\cC\xrightarrow{\Hom_{\cC}(\bu, -)}\Vect \label{eq:ev} \\
\coev&\colon \Vect\xrightarrow{(-)\otimes\bu}\cC\xrightarrow{T^\R}\cC\otimes \cC. \label{eq:coev}
\end{align}
establish self-duality of $\cC$ as an object of the symmetric monoidal bicategory $\PrL$.
\label{thm:cprigidselfdual}
\end{thm}
\begin{proof}
See \cite[Proposition 2.16]{HSSS} for an analogous statement on the level of $\infty$-categories.
\end{proof}

\begin{remark}
The conclusion of the theorem remains true if we drop the assumption that the unit of $\cC$ is compact and projective and replace $\Hom_{\cC}(\bu, -)\colon \cC\rightarrow \Vect$ by the colimit-preserving functor which coincides with $\Hom_{\cC}(\bu, -)$ on compact projective objects.
\end{remark}

\begin{cor}
Let $\cC$ be a cp-rigid monoidal category with a compact projective unit and $\cD$ any monoidal category. Then the functor
\begin{equation}
\cD\otimes \cC\longrightarrow \Fun^\L(\cC, \cD)
\label{eq:cprigidfunctors}
\end{equation}
given by
\[d\boxtimes c\mapsto (c'\mapsto \ev(c, c')\otimes d)\]
is an equivalence.
\end{cor}

\subsection{Duoidal categories}

Let us now study monoidal properties of the equivalence \eqref{eq:cprigidfunctors}. The functor category $\Fun^\L(\cC, \cD)$ has a natural monoidal structure given by the Day convolution \cite{Day} defined by
\begin{equation}
(F\otimes_{\Day} G)(x) = \int^{x_1,x_2\in\cC^{\cp}} \Hom_{\cC}(x_1\otimes x_2, x)\otimes F(x_1)\otimes G(x_2)
\label{eq:Dayconvolution}
\end{equation}
with the unit functor
\[x\mapsto \Hom_\cC(\bu_\cC, x)\otimes \bu_\cD.\]

\begin{prop}
The equivalence \eqref{eq:cprigidfunctors} upgrades to a monoidal equivalence
\[\cD\otimes\cC^{\mop}\xrightarrow{\sim} \Fun^\L(\cC, \cD),\]
where we equip $\Fun^\L(\cC, \cD)$ with the Day convolution monoidal structure.
\label{prop:Dayconvolutionbimodules}
\end{prop}
\begin{proof}
Clearly, the units are compatible since $\bu_\cD\boxtimes \bu_\cC$ is sent to the functor $(x\mapsto \ev(\bu_\cC, x)\otimes \bu_\cD)$.

Now consider two objects $d_1\boxtimes c_1, d_2\boxtimes c_2\in \cD\otimes \cC$. Their Day convolution is computed by
\begin{align*}
&((d_1\boxtimes c_1)\otimes_{\Day} (d_2\boxtimes c_2))(x) \\
= &\int^{x_1,x_2\in\cC^\cp} \Hom_{\cC}(x_1\otimes x_2, x)\otimes (d_1\otimes d_2)\otimes \Hom_\cC(\bu, c_1\otimes x_1)\otimes \Hom_\cC(\bu, c_2\otimes x_2).
\end{align*}

So, we have to exhibit a natural isomorphism
\[\Hom_\cC(\bu, c_2\otimes c_1\otimes x)\cong \int^{x_1,x_2\in\cC^\cp} \Hom_{\cC}(x_1\otimes x_2, x)\otimes \Hom_\cC(\bu, c_1\otimes x_1)\otimes \Hom_{\cC}(\bu, c_2\otimes x_2).\]

By assumption $\cC$ is generated by compact projectives, so it is enough to define this isomorphism on those.

The right-hand side is
\begin{align*}
&\quad \int^{x_1,x_2} \Hom(x_1\otimes x_2, x)\otimes \Hom(\bu, c_1\otimes x_1)\otimes \Hom(\bu, c_2\otimes x_2) \\
&\cong \int^{x_1,x_2} \Hom(x_1\otimes x_2, x)\otimes \Hom(c_1^\vee, x_1)\otimes \Hom(c_2^\vee, x_2) \\
&\cong \Hom(c_1^\vee\otimes c_2^\vee, x) \\
&\cong \Hom(\bu, c_2\otimes c_1\otimes x),
\end{align*}
where we have used \cref{prop:ninjaYoneda} in the third line.
\end{proof}

We will now examine monoidal properties of the self-duality pairings \eqref{eq:ev} and \eqref{eq:coev}.

\begin{prop}
The functors
\[\ev\colon \cC^{\mop}\otimes \cC\longrightarrow \Vect,\qquad \coev\colon \Vect\rightarrow \cC\otimes \cC^{\mop}\]
have a natural lax monoidal structure.
\label{prop:evcoevlaxmonoidal}
\end{prop}
\begin{proof}
We begin with the evaluation functor. The unit map $k\rightarrow \ev(\bu, \bu)=\Hom_\cC(\bu, \bu)$ is given by the inclusion of the identity. Suppose $c_1\boxtimes c_2,d_1\boxtimes d_2\in\cC^{\mop}\otimes \cC$ are two compact projective objects. Then we define $\ev(c_1, c_2)\otimes \ev(d_1, d_2)\rightarrow \ev(d_1\otimes c_1, c_2\otimes d_2)$ via the commutative diagram
\[
\xymatrix{
\ev(c_1, c_2)\otimes \ev(d_1, d_2) \ar[r] \ar@{=}[d] & \ev(d_1\otimes c_1, c_2\otimes d_2) \ar@{=}[d] \\
\Hom(\bu, c_1\otimes c_2)\otimes \Hom(\bu, d_1\otimes d_2) \ar[r] \ar^{\sim}[d] & \Hom(\bu, d_1\otimes c_1\otimes c_2\otimes d_2) \ar^{\sim}[d] \\
\Hom(c_1^\vee, c_2)\otimes \Hom(d_1^\vee, d_2) \ar[r] & \Hom(c_1^\vee\otimes d_1^\vee, c_2\otimes d_2)
}
\]

Next we consider the coevaluation functor. A lax monoidal structure on $\coev$ is the same as an algebra structure on $\coev(k)=T^\R(\bu)$, which, in turn, is provided by \cref{cor:TRalgebra}.
\end{proof}

Now suppose $\cC,\cD$ are cp-rigid monoidal categories with compact projective units and $\cE$ any monoidal category. Then the composition functor
\[\Fun^\L(\cD, \cE)\otimes \Fun^\L(\cC, \cD)\longrightarrow \Fun^\L(\cC, \cE)\]
has a natural lax monoidal structure with respect to the Day convolution.

\begin{prop}
Suppose $\cC,\cD,\cE$ are as above. The diagrams
\[
\xymatrix@C=2cm{
\Fun^\L(\cD, \cE)\otimes \Fun^\L(\cC, \cD) \ar[r] & \Fun^\L(\cC, \cE) \\
\cE\otimes \cD^{\mop}\otimes \cD\otimes \cC^{\mop} \ar^{\id\otimes \ev\otimes \id}[r] \ar^{\sim}[u] & \cE\otimes \cC^{\mop} \ar^{\sim}[u]
}
\]
and
\[
\xymatrix{
\Vect \ar^-{\id}[r] & \Fun^\L(\cC, \cC) \\
\Vect \ar^-{\coev}[r] \ar@{=}[u] & \cC\otimes \cC^{\mop} \ar^{\sim}[u]
}
\]
of lax monoidal functors with respect to the Day convolution commute up to a monoidal natural isomorphism.
\label{prop:Funcompositionlaxmonoidal}
\end{prop}

Recall the following notion (see \cite[Definition 6.1]{AguiarMahajan} where it is called a 2-monoidal category).

\begin{defn}
A \defterm{duoidal category} is a category $\cC$ equipped with two monoidal structures $(\cC, \circ, I)$ and $(\cC, \otimes, J)$, such that the functors $\circ\colon \cC\times \cC\rightarrow \cC$ and $I\colon \Vect\rightarrow \cC$ are lax monoidal with respect to $(\otimes, J)$.
\end{defn}

\begin{example}
Consider the category $\Fun^\L(\cC, \cC)$. It carries a monoidal structure $\circ$ given by the composition of functors whose unit $I$ is the identity functor. It also carries the Day convolution monoidal structure $\otimes_{Day}$. It is shown in \cite[Proposition 50]{GarnerLopezFranco} that the two are compatible so that $\Fun^\L(\cC, \cC)$ is a duoidal category.
\end{example}

\begin{example}
Consider the category $\cC\otimes \cC$. It carries a convolution monoidal structure $\circ$ defined by
\[(M_1\boxtimes M_2)\circ (N_1\boxtimes N_2) = \ev(M_2, N_1)\otimes M_1\boxtimes N_2\]
whose unit is $I = \coev(k)\in\cC\otimes \cC$. It also carries a pointwise monoidal structure $\cC\otimes \cC^{\mop}$ whose unit is $J = \bu_\cC\boxtimes \bu_\cC$. It follows from \cref{prop:evcoevlaxmonoidal} that the two monoidal structures are compatible, so that $\cC\otimes \cC$ becomes a duoidal category. Moreover, $\cC$ is naturally a module category over $\cC\otimes \cC$ with respect to convolution:
\[(\cC\otimes \cC)\otimes \cC\longrightarrow \cC\]
is given by
\[(c_1\boxtimes c_2)\boxtimes d\mapsto \ev(c_2, d)\otimes c_1.\]
\label{ex:convolutionproduct}
\end{example}

We are now ready to relate the two duoidal structures. The following statement combines \cref{prop:Dayconvolutionbimodules,prop:Funcompositionlaxmonoidal}.

\begin{thm}
Suppose $\cC$ is a cp-rigid monoidal category with a compact projective unit. The equivalence \eqref{eq:cprigidfunctors}
\[\cC\otimes \cC\longrightarrow \Fun^\L(\cC, \cC)\]
given by $c_1\boxtimes c_2\mapsto (d\mapsto\ev(c_2, d)\otimes c_1)$ upgrades to an equivalence of duoidal categories, where the two monoidal structures are the convolution product and the pointwise monoidal structure on $\cC\otimes \cC^{\mop}$ while the two monoidal structures on $\Fun^\L(\cC, \cC)$ are the composition of functors and Day convolution. This equivalence intertwines $\cC$ as a $\cC\otimes \cC$-module category with respect to convolution and $\cC$ as a $\Fun^\L(\cC, \cC)$-module category with respect to composition of functors.
\label{thm:endofunctorsduoidal}
\end{thm}

\subsection{Bimodules and lax monoidal functors}

Suppose $f\colon A\rightarrow B$ is a homomorphism of algebras. Then $B$ becomes an $(A, B)$-bimodule with a distinguished element given by $1\in B$. Conversely, the data of an $(A, B)$-bimodule $M$ with a distinguished element $1_M\in M$, such that the action map $B\rightarrow M$ is an isomorphism, is the same as the data of a homomorphism $A\rightarrow B$. In this section we will describe a similar construction on the categorical level. Recall from \cite[Chapters 7.1,7.2]{EGNO} the notion of a module category over a monoidal category.

Suppose $\cC$ and $\cD$ are monoidal categories and $\cM$ a $(\cC, \cD)$ bimodule category together with a distinguished object $\Dist\in\cM$. The action functors of $\cC$ and $\cD$ on $\Dist\in\cM$ define colimit-preserving functors
\[\act_\cC\colon \cC\longrightarrow \cM,\qquad \act_\cD\colon \cD\longrightarrow \cM\]
which we write as $x\mapsto x\otimes \Dist$ and $y\mapsto \Dist\otimes y$, respectively. By the adjoint functor theorem these admit right adjoints that we denote by $\act^\R_\cC$ and $\act^\R_\cD$. The counit of the adjunction defines a natural morphism
\[\epsilon\colon \Dist\otimes\act^\R_\cD(m)\rightarrow m\]
for $m\in\cM$. Moreover, $\act_\cD\colon \cD\rightarrow \cM$ is a functor of right $\cD$-module categories, so $\act^\R_\cD\colon \cM\rightarrow \cD$ is a lax $\cD$-module functor, i.e. we have a natural morphism
\[\phi\colon\act^\R_\cD(m)\otimes y\longrightarrow \act^\R_\cD(m\otimes y)\]
satisfying an associativity axiom.

Consider the functor
\[F_{\cC\cD}=\act^\R_\cD\circ\act_\cC\colon \cC\longrightarrow \cD.\]

\begin{prop}
The morphisms
\[\bu_\cB\rightarrow \act^\R_\cB\circ \act_\cB(\bu_\cB)\cong \act^\R_\cB\circ \act_\cA(\bu_\cA)\]
and
\begin{align*}
\act^\R_\cB(x\otimes\Dist)\otimes\act^\R_\cB(y\otimes\Dist)&\rightarrow \act^\R_\cB(x\otimes \Dist\otimes \act^\R_\cB(y\otimes \Dist)) \\
&\rightarrow \act^\R_\cB(x\otimes y\otimes \Dist)
\end{align*}
define the structure of a lax monoidal functor on $F_{\cA\cB}$.
\label{prop:laxmonoidalactionadjoint}
\end{prop}
\begin{proof}
Let us prove the associativity condition. For brevity denote $a^\R=\act^\R_\cB$, $D=\Dist$. We have to show that the diagram
\[
\xymatrix{
(a^\R(x\otimes D)\otimes a^\R(y\otimes D))\otimes a^\R(z\otimes D) \ar^{\sim}[r] \ar^{\phi}[d] & a^\R(x\otimes D)\otimes (a^\R(y\otimes D)\otimes a^\R(z\otimes D)) \ar^{\phi}[d] \\
a^\R(x\otimes D\otimes a^\R(y\otimes D))\otimes a^\R(z\otimes D) \ar^{\epsilon}[d] & a^\R(x\otimes D)\otimes a^\R(y\otimes D\otimes a^\R(z\otimes D)) \ar^{\epsilon}[d] \\
a^\R((x\otimes y)\otimes D)\otimes a^\R(z\otimes D) \ar^{\phi}[d] & a^R(x\otimes D)\otimes a^\R((y\otimes z)\otimes D) \ar^{\phi}[d] \\
a^\R((x\otimes y)\otimes D\otimes a^\R(z\otimes D)) \ar^{\epsilon}[d] & a^\R(x\otimes D\otimes a^\R((y\otimes z)\otimes D)) \ar^{\epsilon}[d] \\
a^\R((x\otimes y)\otimes z\otimes D) \ar^{\sim}[r] & a^\R(x\otimes (y\otimes z)\otimes D)
}
\]
is commutative. Using naturality and the associativity condition for the lax module structure on $\act^\R$, the above diagram is reduced to
\[
\xymatrix@C=-1cm{
& a^\R(x\otimes D)\otimes (a^\R(y\otimes D)\otimes a^\R(z\otimes D)) \ar_{\phi}[dl] \ar^{\phi}[dr] & \\
a^\R(x\otimes D\otimes a^\R(y\otimes D)\otimes a^\R(z\otimes D)) \ar^{\epsilon}[d] \ar^{\phi}[drr] && a^\R(x\otimes D)\otimes a^\R(y\otimes D\otimes a^\R(z\otimes D)) \ar^{\phi}[d] \\
a^\R((x\otimes y)\otimes D \otimes a^\R(z\otimes D)) \ar^{\epsilon}[dd] && a^R(x\otimes D\otimes a^\R(y\otimes D\otimes a^\R(z\otimes D))) \ar^{\epsilon}[d] \ar^{\epsilon}[ll] \\
 && a^\R(x\otimes D\otimes a^\R((y\otimes z)\otimes D)) \ar^{\epsilon}[d] \\
a^\R((x\otimes y)\otimes z\otimes D) \ar^{\sim}[rr] && a^\R(x\otimes (y\otimes z)\otimes D)
}
\]
The top segment commutes by naturality of $\phi$. The middle segment commutes since $\epsilon$ is a natural transformation of $\cD$-module functors. The bottom segment commutes by naturality of $\epsilon$.

Unitality is proven analogously.
\end{proof}

Note that in the above construction we may freely replace $\cC$ and $\cD$, so we similarly obtain a lax monoidal functor
\[F_{\cD\cC}\colon \cD\longrightarrow \cC.\]

\begin{defn}
Suppose $\cD$ is a monoidal category and $\cM$ a $\cD$-module category with a distinguished object. $\cM$ is \defterm{free of rank 1} if the action functor $\act_\cD\colon \cD\rightarrow \cM$ is an equivalence.
\label{def:freerank1}
\end{defn}

\begin{prop}
Suppose $\cM$ is free of rank 1 over $\cD$. Then the lax monoidal functor $F_{\cC\cD}\colon \cC\rightarrow \cD$ is strongly monoidal and it preserves colimits.
\label{prop:freerank1monoidal}
\end{prop}
\begin{proof}
Since $\act_\cD$ is an equivalence, both the counit $\epsilon\colon \Dist\otimes \act^\R_\cD(m)\rightarrow m$ and the structure of a lax module functor $\phi\colon \act^\R_\cD(m)\otimes y\rightarrow \act^\R_\cD(m\otimes y)$ are isomorphisms. In particular, $F_{\cC\cD}$ is strongly monoidal.

Moreover, $\act_\cD^\R$ is the inverse to $\act_\cD$, so it preserves colimits.
\end{proof}

\subsection{Tannaka reconstruction for bialgebras}
\label{sect:TannakaBialgebras}

Recall Tannaka reconstruction results for bialgebras; we refer to \cite{Deligne,Saavedra} for the commutative case and \cite{UlbrichFibre,UlbrichHopf,SchauenburgTannaka} for the general case.

Let $B\in\Vect$ be a bialgebra. Then $\cC=\CoMod_B$, the category of (left) $B$-comodules, is locally presentable. Moreover, it is equipped with a conservative and colimit-preserving monoidal forgetful functor $F\colon \cC\rightarrow \Vect$ which admits a colimit-preserving right adjoint $F^\R\colon \Vect\rightarrow \cC$ sending $V$ to the cofree $B$-comodule $B\otimes V$ cogenerated by $V$. There is a converse to this statement.

\begin{prop}
Suppose $\cC$ is a monoidal category with a colimit-preserving monoidal forgetful functor $F\colon \cC\rightarrow \Vect$ which admits a colimit-preserving right adjoint $F^\R\colon \Vect\rightarrow \cC$. Then $B=FF^\R(k)$ is a bialgebra and $F$ factors as
\[\cC\longrightarrow \CoMod_B.\]
Moreover, the latter functor is an equivalence if, and only if, $F$ is conservative and preserves equalizers.
\label{prop:tannakabialgebras}
\end{prop}

\begin{remark}
A more familiar statement of Tannaka reconstruction is obtained by passing to compact objects in the above statement. Namely, for a small abelian monoidal category $\cC^\mathrm{c}$ with a biexact tensor product and a monoidal functor
\[F\colon \cC^\mathrm{c}\longrightarrow \Vec\]
to the category of finite-dimensional vector spaces there is a canonical bialgebra $B$ (the bialgebra of coendomorphisms of $F$, see \cite[Section 1.10]{EGNO}), such that $F$ factors through
\[\cC^\mathrm{c}\longrightarrow \CoMod^{\fd}_B\]
through the category of finite-dimensional $B$-comodules. Moreover, the latter functor is an equivalence if, and only if, $F$ is exact and faithful. We refer to \cite[Section 5.4]{EGNO} for more details.
\end{remark}

Let us now be more explicit. Consider the setup of \cref{prop:tannakabialgebras}, where $\cC$ is a monoidal category with enough compact projectives and a compact projective unit. Since $F^\R$ preserves colimits, $F$ preserves compact projective objects. In particular, for $y\in\cC^{\cp}$ the vector space $F(y)$ is finite-dimensional. So, by \cref{prop:rightadjoint} the bialgebra $B$ is
\begin{equation}
B = \int^{y\in\cC^{\cp}} F(y)^\vee\otimes F(y).
\label{eq:bialgebracoend}
\end{equation}
For $y\in\cC^{\cp}$ let us denote by
\[\pi_y\colon F(y)^\vee \otimes F(y)\rightarrow B\]
the natural projection. For $y,z\in\cC$ denote by
\[J_{y, z}\colon F(y)\otimes F(z)\xrightarrow{\sim} F(y\otimes z)\]
the monoidal structure on $F$ (the unit isomorphism will be implicit). The bialgebra structure on $B$ is given on generators as follows:
\begin{itemize}
\item The coproduct is
\[F(y)^\vee\otimes F(y)\xrightarrow{\id\otimes\coev_{F(y)}\otimes \id}F(y)^\vee\otimes F(y)\otimes F(y)^\vee\otimes F(y)\xrightarrow{\pi_y\otimes \pi_y} B\otimes B.\]
\item The counit is
\[F(y)^\vee\otimes F(y)\xrightarrow{\ev_{F(y)}} k.\]
\item The product is
\begin{align*}
(F(y)^\vee\otimes F(y))\otimes (F(z)^\vee\otimes F(z))&\cong (F(y)\otimes F(z))^\vee\otimes F(y)\otimes F(z)\\
&\xrightarrow{(J_{y,z}^{-1})^\vee\otimes J_{y,z}} F(y\otimes z)^\vee\otimes F(y\otimes z)\\
&\xrightarrow{\pi_{y\otimes z}} B.
\end{align*}
\item The unit is
\[k\cong F(\bu)^\vee\otimes F(\bu)\xrightarrow{\pi_\bu} B.\]
\end{itemize}

It will also be useful to think about $\pi_y$ as elements
\[T_y\in B\otimes \End(F(y)).\]
The following statement is immediate from the above formulas.

\begin{thm}
The bialgebra $B$ is spanned, as a $k$-vector space, by the matrix coefficients of $T_y$ for $y\in\cC^{\cp}$, subject to the relation
\begin{equation}
F(f) \circ T_x = T_y \circ F(f)
\label{eq:explicitcoend}
\end{equation}
for every $f\colon x\rightarrow y$. Moreover:
\begin{itemize}
\item For $y\in\cC^{\cp}$ we have
\begin{equation}
\Delta(T_y) = T_y\otimes T_y.
\end{equation}

\item For $y\in\cC^{\cp}$ we have
\begin{equation}
\epsilon(T_y) = \id_{F(y)}\in\End(F(y)).
\end{equation}

\item Suppose $x,y\in\cC^{\cp}$ are two objects. Then
\begin{equation}
J^{-1}_{x,y}T_{x\otimes y}J_{x,y}=(T_x\otimes \id_{F(y)})(\id_{F(x)}\otimes T_y)
\label{eq:tensorproducttt}
\end{equation}
as elements of $B\otimes \End(F(x\otimes y))\cong B\otimes \End(F(x)\otimes F(y))$.

\item $T_\bu\in B\otimes \End(F(\bu))\cong B$ is the unit.
\end{itemize}
\label{thm:Bmatrixpresentation}
\end{thm}

Let us now study what happens when $\cC$ is in addition equipped with a braiding.

\begin{defn}
Suppose $\cC$ is a braided monoidal category and $F\colon \cC\rightarrow \Vect$ a monoidal functor. For $x,y\in\cC$ the \defterm{$R$-matrix} is
\[R_{x,y}\colon F(x)\otimes F(y)\xrightarrow{J_{x,y}} F(x\otimes y)\xrightarrow{F(\sigma_{x, y})} F(y\otimes x)\xrightarrow{J_{y,x}^{-1}} F(y)\otimes F(x) \xrightarrow{\sigma_{F(x), F(y)}^{-1}} F(x)\otimes F(y).\]
\label{def:Rmatrix}
\end{defn}

It will be convenient to use the standard matrix notation for $R$-matrices acting on several variables: given $x,y,z\in\cC$ we denote
\[R_{12} = R_{x,y}\otimes \id\]
as an element of $\End(F(x)\otimes F(y)\otimes F(z))$ and similarly for $R_{13}$ and $R_{23}$. We let the transposed $R$-matrix $R_{21}$ be
\[F(x)\otimes F(y)\xrightarrow{\sigma^{-1}} F(y)\otimes F(x)\xrightarrow{R_{y, x}} F(y)\otimes F(x)\xrightarrow{\sigma} F(x)\otimes F(y).\] We also denote
\[T_1 = T_x\otimes \id,\qquad T_2 = \id\otimes T_y\]
as elements of $B\otimes \End(F(x)\otimes F(y))$.

\begin{prop}
Suppose $x,y,z\in\cC$. Then the $R$-matrix satisfies the Yang--Baxter equation
\begin{equation}
R_{12}R_{13}R_{23} = R_{23}R_{13}R_{12}
\label{eq:YangBaxter}
\end{equation}
in $\End(F(x)\otimes F(y)\otimes F(z))$. Moreover, $T$ satisfies the FRT relation
\begin{equation}
R_{12}T_1T_2 = T_2T_1 R_{12}
\label{eq:RTT}
\end{equation}
in $B\otimes \End(F(x)\otimes F(y))$.
\label{prop:YBRTTequations}
\end{prop}
\begin{proof}
Denote
\[\check{R}_{x,y}\colon F(x)\otimes F(y)\xrightarrow{J_{x,y}} F(x\otimes y)\xrightarrow{F(\sigma_{x, y})} F(y\otimes x)\xrightarrow{J_{y,x}^{-1}} F(y)\otimes F(x).\]
Then the Yang--Baxter equation \eqref{eq:YangBaxter} is equivalent to the braid equation
\[\check{R}_{12}\check{R}_{23}\check{R}_{12} = \check{R}_{23}\check{R}_{12}\check{R}_{23}\]
which holds in any braided monoidal category.

By \eqref{eq:explicitcoend}, we have $F(\sigma_{x,y})T_{x\otimes y} = T_{y\otimes x} F(\sigma_{x,y})$. Relation \eqref{eq:tensorproducttt} and the equality 
\[ F(\sigma_{x,y}) = J_{y,x} \hat{R}_{x,y} J^{-1}_{x,y} \]
imply \eqref{eq:RTT}. 
\end{proof}

\begin{remark}
Quantum groups were originally introduced in \cite{FST,FRT} as bialgebras as in \cref{thm:Bmatrixpresentation} satisfying the FRT relation \eqref{eq:RTT}. The above statements show, conversely, that this relation naturally follows from the categorical framework.
\end{remark}

\subsection{Coend algebras and reflection equation}
\label{sect:coend}

Let $\cC$ be a cp-rigid monoidal category. Recall the formula for the right adjoint $T^\R\colon \cC\rightarrow \cC\otimes \cC$ for the tensor product functor $\cC\otimes\cC\rightarrow \cC$ from \cref{prop:TRformula}.

\begin{defn}
The \defterm{canonical coend} is the object $\cF\in\cC$ defined by
\begin{equation}
\cF = T T^\R(\bu) = \int^{x\in\cC^{\cp}} x^\vee\otimes x.
\label{eq:REAcoend}
\end{equation}
\label{def:coend}
\end{defn}

For $x\in \cC^{\cp}$ let us denote by
\[\pi_x\colon x^\vee\otimes x\rightarrow \cF\]
the natural projection.

Now, suppose in addition that $\cC$ is braided monoidal. Then $\cF$ admits a structure of a braided Hopf algebra (see e.g. \cite{LyubashenkoMajid,Lyubashenko,ShimizuCoends}). Explicitly, the algebra structure is given on generators as follows:
\begin{itemize}
\item The product is
\begin{align*}
(x^\vee\otimes x)\otimes (y^\vee\otimes y)&\xrightarrow{\sigma_{x^\vee\otimes x, y^\vee}} y^\vee\otimes x^\vee\otimes x\otimes y\\
&\cong (x\otimes y)^\vee\otimes x\otimes y\\
&\xrightarrow{\pi_{x\otimes y}}\cF.
\end{align*}

\item The unit is
\[\bu\xrightarrow{\pi_\bu} \cF.\]
\end{itemize}

Consider a monoidal functor $F\colon \cC\rightarrow \Vect$. The projections $\pi_x$ give rise to elements
\[K_x\in F(\cF)\otimes \End(F(x)).\]

Comparing the formulas \eqref{eq:bialgebracoend} and \eqref{eq:REAcoend}, we see that there is an isomorphism of vector spaces
\[F(\cF)\cong B.\]
In particular, as before, $F(\cF)$ is spanned, as a $k$-vector space, by the matrix coefficients of $K_x$ for $x\in\cC^{\cp}$ subject to the relation \eqref{eq:explicitcoend} for every $f\colon x\rightarrow y$. As before, $K_\bu\in F(\cF)$ is the unit. However, the multiplication is different. The following was proved in \cite{MajidFoundations,DoninKulishMudrov}.

\begin{prop}
Suppose $x,y\in\cC^{\cp}$ are two objects. Then the reflection equation
\begin{equation}
R_{21}K_1R_{12}K_2 = K_2R_{21}K_1R_{12}
\label{eq:reflectionequation}
\end{equation}
holds in $F(\cF)\otimes \End(F(x)\otimes F(y))$.
\end{prop}

\begin{remark}
The reflection equation algebra in the theory of quantum groups was introduced in \cite{KulishSklyanin} as the algebra generated by the matrix elements of $K$ satisfying the reflection equation \eqref{eq:reflectionequation}. We see that it coincides with $F(\cF)$. So, $\cF$ is also sometimes known as the reflection equation algebra.
\end{remark}

\begin{example}
Suppose $H$ is a Hopf algebra and consider $\cC=\LMod_H$. Then the coend algebra $\cF$ is a Drinfeld twist of the restricted dual Hopf algebra
\[H^\circ=\int^{V\in\LMod_H^{\cp}} V^\vee\otimes V,\]
see \cite[Definition 4.12]{DoninMudrovREA}.
\end{example}

\section{Harish-Chandra bimodules}

In this section we study categories of classical and quantum Harish-Chandra bimodules as well as introduce Harish-Chandra bialgebroids.

\subsection{General definition}
\label{sect:HCgeneral}

We will now present a general categorical definition which encompasses categories of both classical and quantum Harish-Chandra bimodules. We refer to \cref{sect:classicalHC} for a relationship to the usual Harish-Chandra bimodules. This formalism is closely related to the theory of dynamical extensions of monoidal categories introduced in \cite{DoninMudrovVectorBundles}, see \cref{rmk:DoninMurov}.

Throughout this section we fix a cp-rigid monoidal category $\cC$. Recall from \cite[Definition 7.13.1]{EGNO} that the Drinfeld center $\ZDr(\cC)$ is the braided monoidal category consisting of pairs $(z, \tau)$, where $z\in\cC$ and
\[\tau_x\colon x\otimes z\xrightarrow{\sim} z\otimes x\]
is a natural isomorphism satisfying standard compatibilities. The monoidal structure is given by
\[(z, \tau)\otimes (z', \tau') = (z\otimes z', \tilde{\tau}),\]
where $\tilde{\tau}$ is the composite
\[x\otimes z\otimes z'\xrightarrow{\tau_x\otimes \id_{z'}} z\otimes x\otimes z'\xrightarrow{\id\otimes \tau'_x} z\otimes z'\otimes x,\]
where we omit associators. We refer to \cite[Proposition 8.5.1]{EGNO} for the braided monoidal structure on $\ZDr(\cC)$.

\begin{defn}
Let $(\cL, \tau)$ be a commutative algebra in $\ZDr(\cC)$. The \defterm{category of Harish-Chandra bimodules} is
\[\HC(\cC, \cL) = \LMod_\cL(\cC).\]
When there is no confusion, we simply denote $\HC=\HC(\cC, \cL)$.
\end{defn}

\begin{remark}
A commutative algebra in the Drinfeld center is called a \emph{base algebra} in \cite[Definition 4.1]{DoninMudrovVectorBundles}. The full subcategory of $\HC(\cC, \cL)$ consisting of free left $\cL$-modules is called a \emph{dynamical extension} of $\cC$ over $\cL$ in \cite[Section 4.2]{DoninMudrovVectorBundles}.
\label{rmk:DoninMurov}
\end{remark}

If $H$ is a Hopf algebra, recall that the Drinfeld center $\ZDr(\LMod_H)$ is equivalent to the category of Yetter--Drinfeld modules over $H$ (see \cite[Section XIII.5]{Kassel}). This gives rise to the following important example.

\begin{prop}
Suppose $H$ is a Hopf algebra and consider $\cC=\LMod_H$. A commutative algebra $\cL$ in $\ZDr(\LMod_H)$ is the same as an $H$-algebra $\cL$ equipped with a left $H$-coaction $\delta\colon\cL\rightarrow H\otimes \cL$, a map of $H$-algebras, denoted by $x\mapsto x_{(-1)}\otimes x_{(0)}$ satisfying
\[xy = y_{(0)} (S^{-1}(y_{(-1)}))\triangleright x,\qquad x,y\in \cL.\]
The corresponding isomorphism $\tau_M\colon M\otimes \cL\rightarrow \cL\otimes M$ is given by
\[m\otimes x\mapsto x_{(0)}\otimes (S^{-1}(x_{(-1)})\triangleright m).\]
\label{prop:YDmodules}
\end{prop}
\begin{proof}
The compatibility of $\tau_M$ with the monoidal structure on $\LMod_H$ follows from the coassociativity and counitality of the $H$-coaction. The compatibility of $\tau_M$ with the algebra structure on $\cL$ is equivalent to the equation
\[
x_{(0)}y_{(0)}\otimes S^{-1}(y_{(-1)}) S^{-1}(x_{(-1)}) \triangleright m = (xy)_{(0)}\otimes S^{-1}((xy)_{(-1)})\triangleright m,
\]
which follows from the condition that $\cL\rightarrow H\otimes \cL$ is an algebra map. The commutativity of the multiplication on $\cL\in\ZDr(\LMod_H)$ is
\[xy = y_{(0)} (S^{-1}(y_{(-1)}))\triangleright x.\]
\end{proof}

\begin{remark}
The inverse morphism $\cL\otimes M\rightarrow M\otimes \cL$ is given by
\[x\otimes m\mapsto x_{(-1)}\triangleright m\otimes x_{(0)}.\]
\label{rmk:inversefieldgoal}
\end{remark}

\begin{example}
Consider a Hopf algebra $H$ and let $\cL=H$.  Consider the adjoint action of $H$ on $\cL$:
\[h\otimes x\mapsto h_{(1)} x S(h_{(2)})\]
for $h\in H$ and $x\in\cL$. Consider the $H$-coaction $\cL\rightarrow H\otimes \cL$ given by the coproduct on $H$. Then
\[S^{-1}(y)\triangleright x = S^{-1}(y_{(2)})xy_{(1)}.\]
In particular,
\begin{align*}
y_{(2)} (S^{-1}(y_{(1)}))\triangleright x &= y_{(3)} S^{-1}(y_{(2)})x y_{(1)} \\
&= \epsilon(y_{(2)}) x y_{(1)} \\
&= xy,
\end{align*}
which shows that $(\cL, \tau)$ is a commutative algebra in $\ZDr(\LMod_H)$.
\label{ex:hopfalgebraadjoint}
\end{example}

Since $\cL$ is a commutative algebra in $\ZDr(\cC)$, the category $\HC$ has a natural monoidal structure given by the relative tensor product: given left $\cL$-modules $M,N\in\cC$, we may turn $M$ into a right $\cL$-module using $\tau_M$ and then the tensor product is given by $M\otimes_\cL N$. We also have an adjunction
\[
\adj{\free\colon \cC}{\HC\colon \forget,}
\]
where $\free\colon\cC\rightarrow \HC$  is the monoidal functor $x\mapsto \cL\otimes x$ given by the free left $\cL$-module and $\forget\colon\HC\rightarrow \cC$ is given by forgetting the $\cL$-module structure.

Observe that $\cL^{\op}$ is an algebra in $\cC^{\mop}$. Moreover, it lifts to a commutative algebra in $\ZDr(\cC^{\mop})$ if we consider the inverse isomorphism $\tau_x$.

\begin{lm}
There is a natural monoidal equivalence $\HC(\cC, \cL)^{\mop}\cong \HC(\cC^{\mop}, \cL^{\op})$.
\label{lm:HCop}
\end{lm}

The following construction explains why $\HC$ deserves to be called the category of bimodules. There is a natural monoidal functor
\begin{equation}
\bimod\colon \HC\longrightarrow \BiMod{\cL}{\cL}(\cC)
\label{eq:bimod}
\end{equation}
given by sending a left $\cL$-module $M$ to the $\cL$-bimodule, where the right $\cL$-action is obtained via $\tau_M$. It realizes $\HC$ as a full subcategory of $\BiMod{\cL}{\cL}(\cC)$ consisting of objects $M\in\BiMod{\cL}{\cL}(\cC)$ such that the right and left actions are related by $\tau_M$.

Let us now analyze categorical properties of $\HC$.

\begin{prop}
The category $\HC$ is cp-rigid. Moreover, we may take $\free(V)\in\HC$ for all $V\in\cC^{\cp}$ as the generating set of compact projective objects. If the unit of $\cC$ is compact projective, so is the unit in $\HC$.
\label{prop:HCcprigid}
\end{prop}
\begin{proof}
The functor $\free\colon \cC\rightarrow \HC$ has a colimit-preserving right adjoint $\forget\colon \HC\rightarrow \cC$. So, $\free(V)\in\HC$ is compact projective if $V\in\cC^{\cp}$.

The category $\HC$ is generated by $\free(V)$ for $V\in\cC$ since $\forget$ is conservative. But since $\cC$ has enough compact projectives, we may restrict to $V\in\cC^{\cp}$.

Since $\cC$ is cp-rigid, the objects $V\in\cC^{\cp}$ are dualizable. Since $\free\colon \cC\rightarrow \HC$ is monoidal, the objects $\free(V)\in\HC$ are also dualizable. But we have just shown that such objects are the generating compact projective objects, while by \cite[Proposition 4.1]{BJS} it is enough to check cp-rigidity on the generating compact projective objects.

The unit of $\HC$ is $\cL$ viewed as a free left $\cL$-module of rank 1, so
\[\Hom_{\HC}(\cL, -)\cong \Hom_{\cC}(\bu_\cC, \forget(-))\]
which shows that $\cL$ is compact projective if, and only if, $\bu_\cC\in\cC$ is.
\end{proof}

\subsection{Quantum moment maps}

Recall that for an algebra $A\in\Rep(G)$ a quantum moment map is a map $\mu\colon\U\g\rightarrow A$ such that the infinitesimal $\g$-action on $A$ is given by $[\mu(x), -]$ for $x\in\g$. The following version of this definition in our setting was introduced in \cite[Definition 3.1]{SafronovQMM}.

\begin{defn}
Let $A\in\cC$ be an algebra. A \defterm{quantum moment map} is an algebra map $\mu\colon \cL\rightarrow A$ such that the diagram
\begin{equation}
\xymatrix@R=0.5cm{
\cL\otimes A \ar^{\mu\otimes \id}[r] & A\otimes A \ar^{m}[dr] & \\
&& A \\
A\otimes \cL \ar^{\tau_A}[uu] \ar^{\id\otimes \mu}[r] & A\otimes A \ar_{m}[ur]
}
\label{eq:quantummomentmap}
\end{equation}
commutes.
\label{def:quantummomentmap}
\end{defn}

\begin{remark}
Recall that $\cL\in\ZDr(\cC)$ is a commutative algebra. The quantum moment map condition expressed by \eqref{eq:quantummomentmap} says that $\mu\colon \cL\rightarrow A$ is a central map.
\end{remark}

\begin{prop}
An algebra in $\HC$ is an algebra in $\cC$ equipped with a quantum moment map.
\label{prop:HCalgebras}
\end{prop}
\begin{proof}
Via the embedding $\bimod\colon \HC\rightarrow \BiMod{\cL}{\cL}(\cC)$ of \eqref{eq:bimod} an algebra $A\in\HC$ gives rise to an algebra in $\BiMod{\cL}{\cL}(\cC)$. An algebra in the category of bimodules is the same as an algebra $A\in \cC$ equipped with an algebra map $\mu\colon\cL\rightarrow A$. The condition that it lands in $\HC\subset \BiMod{\cL}{\cL}(\cC)$ is precisely the quantum moment map equation \eqref{eq:quantummomentmap}.
\end{proof}

The following is \cite[Definition 3.10]{SafronovQMM}.

\begin{defn}
Suppose $\epsilon\colon \cL\rightarrow \bu_\cC$ is a morphism of algebras in $\cC$ and $A$ is an algebra equipped with a quantum moment map. The \defterm{Hamiltonian reduction of $A$} is
\[\Hom_{\cC}(\bu_\cC, A\otimes_\cL \bu)\cong \Hom_{\LMod_A(\cC)}(A\otimes_\cL \bu_\cC, A\otimes_\cL \bu_\cC).\]
\label{def:Hamiltonianreduction}
\end{defn}

A canonical example of an algebra with a quantum moment map we will use is the following. Let $T_{\HC}\colon \HC\otimes \HC\rightarrow \HC$ be the tensor product functor. By \cref{prop:evcoevlaxmonoidal} the object $T^\R(\bu_{\HC})\in\HC\otimes \HC^{\mop}$ is an algebra. Identifying $\HC(\cC, \cL)^{\mop}\cong \HC(\cC^{\mop}, \cL^{\op})$ using \cref{lm:HCop}, we see that
\[(\forget\otimes \forget)(T_{\HC}^\R(\bu_{\HC}))\in\cC\otimes \cC^{\mop}\]
is an algebra equipped with a quantum moment map from $\cL\boxtimes \cL^{\op}$.

\begin{defn}
Let $\cC, \HC$ be as before. The algebra $\cD\in\cC\otimes \cC^{\mop}$ is
\[\cD = (\forget\otimes \forget)(T_{\HC}^\R(\bu_{\HC})).\]
\label{def:Heisenbergdouble}
\end{defn}

We denote the canonical quantum moment map by
\begin{equation}
\mu\colon \cL\boxtimes \cL^{\op}\longrightarrow T_\cC^\R(\cL).
\label{eq:idmomentmap}
\end{equation}

\begin{prop}
We have an equivalence
\[\cD\cong \int^{x\in\cC^{\cp}} (\cL\otimes x^\vee)\boxtimes x\cong \int^{x\in\cC^{\cp}} x^\vee\boxtimes (x\otimes \cL),\]
where the latter isomorphism is provided by \cref{prop:TRmodulefunctor}. The algebra structure is given by
\begin{align*}
((\cL\otimes x^\vee)\boxtimes x)\otimes((\cL\otimes y^\vee)\boxtimes y)&\cong (\cL\otimes x^\vee\otimes \cL\otimes y^\vee)\boxtimes (y\otimes x) \\
&\xrightarrow{\id\otimes \tau_{x^\vee}\otimes \id} (\cL\otimes \cL\otimes x^\vee\otimes y^\vee)\boxtimes (y\otimes x) \\
&\xrightarrow{m\otimes \id} (\cL\otimes (y\otimes x)^\vee)\boxtimes (y\otimes x) \\
&\xrightarrow{\pi_{y\otimes x}} \int^{x\in\cC^{\cp}} (\cL\otimes x^\vee)\boxtimes x
\end{align*}
The two quantum moment maps $\cL,\cL^{\op}\rightarrow \cD$ are given by
\[\cL\cong (\cL\otimes \bu)\boxtimes \bu\xrightarrow{\pi_\bu} \cD\]
and
\[\cL^{\op}\cong \bu\boxtimes (\bu\otimes \cL)\xrightarrow{\pi_\bu} \cD.\]
\end{prop}
\begin{proof}
Since $\free\colon \cC\rightarrow \HC$ is a monoidal functor, by adjunction we get a natural isomorphism
\begin{equation}
(\forget\otimes \forget) \circ T_{\HC}^\R\cong T_\cC^\R\circ \forget,
\label{eq:idharishchandraiso}
\end{equation}
where $T_\cC\colon \cC\otimes \cC\rightarrow \cC$ is the tensor product functor. In particular, applying \eqref{eq:idharishchandraiso} to $\bu_{\HC}$, we get isomorphisms
\[\cD\cong \int^{x\in\cC^{\cp}} (\cL\otimes x^\vee)\boxtimes x\cong \int^{x\in\cC^{\cp}} x^\vee\boxtimes (x\otimes \cL)\]
using \cref{prop:TRformula}.

We have a natural isomorphism
\[(\forget\circ\free\otimes \id)\circ T_\cC^\R(-)\cong T_\cC^\R(\forget\circ\free(-))\]
given by \cref{prop:TRmodulefunctor} which gives rise to an algebra isomorphism
\[\cD\cong(\forget\circ\free\otimes\id)\circ T_\cC^\R(\bu)\]
which gives the required formula.
\end{proof}

\begin{example}
Suppose $H$ is a Hopf algebra, $\cL$ is a commutative algebra in $\ZDr(\LMod_H)$ (see \cref{prop:YDmodules}) and $\cC=\LMod_H$. Let
\[H^\circ = \int^{V\in\LMod_H^{\cp}} V^\vee\otimes V\]
is the restricted dual Hopf algebra. By construction $\cL$ is an $H$-comodule algebra and $H^\circ$ is an $H$-module algebra (via the left $H$-action). Then $\cD$ is the smash product algebra generated by $\cL$ and $H^\circ$ with the additional relation
\[hl = l_{(0)} (S^{-1}(l_{(-1)})\triangleright h)\]
for $h\in H^\circ$ and $l\in\cL$.
\end{example}

\subsection{Classical Harish-Chandra bimodules}
\label{sect:classicalHC}

Let $G$ be a reductive group over a characteristic zero field $k$ and denote by $\g$ its Lie algebra. Let $\Rep(G)$ be the ind-completion of the category of finite-dimensional representations. The category $\Rep(G)$ is semisimple, so it has enough compact projectives and its unit is compact projective.

Suppose $V\in\Rep(G)$ is a $G$-representation. For $x\in\U\g$ and $v\in V$ we denote by $x\triangleright v$ the induced $\U\g$-action on $V$. Consider the natural isomorphism
\begin{equation}
\tau_V\colon V\otimes \U\g\longrightarrow \U\g\otimes V
\label{eq:classicalfieldgoal}
\end{equation}
given by
\[
v\otimes x\mapsto x\otimes v - 1\otimes xv
\]
for $x\in\g$. It follows from \cref{prop:YDmodules} that $(\U\g, \tau)$ defines a commutative algebra in $\ZDr(\Rep(G))$.

\begin{defn}
The \defterm{category of classical Harish-Chandra bimodules} is
\[\HC(G) = \HC(\Rep(G), \U\g).\]
\label{def:classicalHC}
\end{defn}

\begin{remark}
The embedding \eqref{eq:bimod} realizes $\HC(G)$ as the category of $\U\g$-bimodules whose diagonal $\g$-action is integrable (see \cite[Definition 5.2]{BernsteinGelfand} for the original definition of Harish-Chandra bimodules).
\end{remark}

The following easy lemma (see \cite[Example 3.4]{SafronovQMM}) shows that the definition of a quantum moment map we gave in \cref{def:quantummomentmap} coincides with the classical notion of a quantum moment map.

\begin{lm}
Let $A\in\Rep(G)$ be an algebra. A quantum moment map $\mu\colon\U\g\rightarrow A$ is the same as an algebra map such that for every $x\in\g$ the commutator $[\mu(x), -]$ coincides with the differential of the $G$-action.
\end{lm}

In the same way, the quantum Hamiltonian reduction from \cref{def:Hamiltonianreduction} coincides with the usual definition
\[A\ham G = (A/A\mu(\g))^G\]
of the reduced algebra.

Given a variety $X$ equipped with a $G$-action, the algebra of differential operators $\D(X)$ carries a quantum moment map $\mu\colon \U\g\rightarrow \D(X)$ which sends $\g\subset \U\g$ to vector fields on $X$ generating the infinitesimal action. For instance, $\D(G)$ carries a moment map
\begin{equation}
\mu\colon \U\g\otimes \U\g^{\op}\rightarrow \D(G)
\label{eq:DGmomentmap}
\end{equation}
coming from the left and right $G$-action on itself. Let us explain how it arises in our context.

By the Peter--Weyl theorem we have an isomorphism of algebras
\[\cO(G)\cong \int^{V\in\Rep^{\fd}(G)} V^\vee\boxtimes V\in\Rep(G)\otimes \Rep(G),\]
where $\cO(G)$ carries a $G\times G$-action coming from the left and right $G$-action on itself. Using this we can also describe the algebra $\cD$ from \cref{def:Heisenbergdouble}.

\begin{prop}
The algebra $\cD\in\HC(G)\otimes \HC(G)^{\mop}$ is isomorphic to $\D(G)\cong\U\g\otimes \cO(G)$ equipped with the $G\times G$-action and the quantum moment map \eqref{eq:DGmomentmap}.
\end{prop}

In the abelian case the category of Harish-Chandra bimodules has a straightforward description. Suppose $H$ is a split torus; let $\h$ be its Lie algebra and $\Lambda=\Hom(H, \Gm)$ the character lattice. Then $\Rep(H)$ is equivalent to the category of $\Lambda$-graded vector spaces and $\HC(H)$ is equivalent to the category of $\Lambda$-graded $\Sym(\h)$-modules $\oplus_{\lambda\in\Lambda} M(\lambda)$.

Given $\lambda\in\Lambda$ we consider the translation functor $\lambda^*\colon \LMod_{\Sym(\h)}\rightarrow \LMod_{\Sym(\h)}$. Then the monoidal structure $\otimes^{\HC}$ on $\HC(H)$ is given by
\[M\otimes^{\HC} N = \bigoplus_{\lambda\in\Lambda} \lambda^*(M)\otimes N(\lambda).\]

Suppose $V\in\Rep(H)$. Given a vector $v\in V$ of weight $\mu\in\Lambda$ and $f\in\cO(\h^*)\cong\U\h$ the map \eqref{eq:classicalfieldgoal} is given by
\[v\otimes f(\lambda)\mapsto f(\lambda - \mu)\otimes v\]
for $\lambda\in\h^*$. It is convenient to write it as
\[v\otimes f(\lambda)\mapsto f(\lambda - h)\otimes v,\]
where $h$ is understood as acting on $v\in V$. Similarly, given a collection of representations $V_1, \dots, V_n\in \Rep(H)$ and vectors $v_i\in V_i$ we denote
\[f(\lambda - h^{(i)}) v_1\otimes \dots v_n = f(\lambda - \mu_i) v_1\otimes \dots v_n\]
if $v_i$ has weight $\mu_i\in\Lambda$.

\subsection{Quantum groups}
\label{sect:quantumgroups}

In this section we fix our conventions for quantum groups. Fix $k=\C$. Let $G$ be a connected reductive group, $B, B_-\subset G$ a pair of opposite Borel subgroups and $H=B\cap B_-$ a Cartan subgroup. Denote by $\Lambda=\Hom(H, \Gm)$ its weight lattice and $\Lambda^\vee=\Hom(\Gm, H)$ the coweight lattice; we denote by $\langle -, -\rangle\colon \Lambda^\vee\times\Lambda\rightarrow \Z$ the canonical pairing. For two simple roots $\alpha_i,\alpha_j\in\Lambda$ denote by $\alpha_i\cdot \alpha_j\in\Z$ the corresponding entry of the symmetrized Cartan matrix. Choose an integer $d\in\Z$ and a symmetric bilinear form $(-, -)\colon \Lambda\times\Lambda\rightarrow \frac{1}{d}\Z$, such that $(\alpha_i, \alpha_j) = \alpha_i\cdot \alpha_j$. Given a complex number $q^{1/d}\in\C^\times$ we have the exponentiated pairing
\[\Pi\colon \Lambda\times \Lambda\longrightarrow \C^\times\]
given by $\lambda,\mu\mapsto q^{-(\lambda, \mu)}$. Our assumption is that $q^{1/d}$ is not a root of unity.

We denote by $\Uq(\g)$ the quantum group defined as in \cite{Lusztig} with a slight modification that its Cartan part is $\Uq(\h)=k[\Lambda]$ with Cartan generators $K_\mu$ for $\mu\in\Lambda$ (note that the Cartan part in \cite{Lusztig} is $k[\Lambda^\vee]$). We denote by $\Uq(\b)\subset \Uq(\g)$ the quantum Borel subalgebra, $\Uq(\n), \Uq(\n_-)\subset \Uq(\g)$ the quantum nilpotent subalgebras and $\Uq^{>0}(\n),\Uq^{<0}(\n_-)$ their augmentation ideals. For each simple root $\alpha$ we denote by $\{E_\alpha, K_\alpha, F_\alpha\}$ the corresponding generators of the $\Uq(\sl_2)$-subalgebra (they are denoted by $E_i, \tilde{K}_i, F_i$ in \cite[Section 3.1.1]{Lusztig}).

We have the corresponding categories obtained from this data:
\begin{itemize}
\item $\Rep_q(H)$ is the braided monoidal category of $\Lambda$-graded vector spaces with the braiding given by $\Pi\tau$, where $\tau$ is the map exchanging the tensor factors.

\item $\Rep_q(G)$ is the ind-completion of the braided monoidal category of finite-dimensional $\Lambda$-graded vector spaces with a $\Uq(\g)$-module structure, such that for every vector $x_\lambda$ of weight $\lambda\in\Lambda$ we have $K_\mu x_\lambda = q^{(\mu, \lambda)} x_\lambda$. The braiding is given by $\Theta\circ\Pi\circ\tau$, where $\Theta\in\Uq(\n_-)\widehat{\otimes}\Uq(\n)$ is the so-called quasi $R$-matrix. We refer to \cite[Section 32]{Lusztig} for more details.

\item $\Rep_q(B)$ is the ind-completion of the monoidal category of finite-dimensional $\Lambda$-graded vector spaces with a compatible $\Uq(\b)$-module structure.
\end{itemize}

\begin{defn}
A $\Uq(\g)$-module $M$ is \defterm{integrable} if it lies in the image of the forgetful functor
\[\Rep_q(G)\longrightarrow \LMod_{\Uq(\g)}.\]
\end{defn}
Equivalently, an integrable $\Uq(\g)$-module is a locally finite type 1 $\Uq(\g)$-module. We introduce an analogous definition for $\Uq(\b)$-modules.

Denote by $\cO_q(G)\in\Rep_q(G)$ the coend algebra from \cref{def:coend}.

\begin{defn}
A $\Uq(\g)$-module $M$ is \defterm{locally finite} if for every $m\in M$ the vector space $\Uq(\g)m$ is finite-dimensional.
\end{defn}

The algebra $\Uq(\g)$ with respect to the adjoint $\Uq(\g)$-action on itself $x, y\mapsto x_{(1)}yS(x_{(2)})$ is not locally finite and we denote by
\[\Uq(\g)^\lf\subset\Uq(\g)\]
the largest locally finite submodule.

\begin{example}
Consider $\Uq(\sl_2)$ with the generators $E,K,F$ and relations
\[KE = q^2 EK,\qquad KF = q^{-2} FK,\qquad EF-FE=\frac{K-K^{-1}}{q-q^{-1}}.\]
Then $\Uq(\sl_2)^\lf$ is the subalgebra generated by $EK^{-1}$, $F$ and $K^{-1}$.
\label{ex:Uqsl2lf}
\end{example}

It is easy to see that $\Uq(\g)^\lf\subset \Uq(\g)$ is a subalgebra, but note that it is not a subcoalgebra. Nevertheless, the following is shown in \cite[Theorem 7.1.6]{Joseph}.

\begin{prop}
$\Uq(\g)^\lf\subset \Uq(\g)$ is a left coideal, i.e. the coproduct restricts to a map
\[
\Delta\colon \Uq(\g)^\lf\longrightarrow\Uq(\g)\otimes \Uq(\g)^{\lf}.
\]
\label{prop:Uqlfcoideal}
\end{prop}

\begin{remark}
There is a close relationship between the algebras $\Uq(\g)^{\lf}$ and $\cO_q(G)$ which can be established using the quantum Killing form \cite{Rosso}. If $G$ is semisimple simply-connected, $\Uq(\g)^\lf\cong \cO_q(G)$, see \cite[Proposition 7.1.23]{Joseph} and \cite[Theorem 2.113]{VoigtYuncken}.
\end{remark}

\subsection{Quantum Harish-Chandra bimodules}
\label{sect:quantumHC}

For $V\in\Rep_q(G)$, $v\in V$ an $x\in\Uq(\g)$ we denote by $x\triangleright v$ the $\Uq(\g)$-action. For $x\in\Uq(\g)^\lf$ we denote by $\Delta(x) = x_{(1)}\otimes x_{(2)}$ the coproduct on $\Uq(\g)$, where we note that $x_{(2)}\in\Uq(\g)^\lf$ by \cref{prop:Uqlfcoideal}. We define the natural isomorphism
\begin{equation}
\tau_V\colon V\otimes \Uq(\g)^\lf\longrightarrow \Uq(\g)^\lf\otimes V
\label{eq:quantumfieldgoal}
\end{equation}
by
\[v\otimes x\mapsto x_{(2)}\otimes S^{-1}(x_{(1)})\triangleright v.\]
Consider $\Uq(\g)^\lf\in\Rep_q(G)$ with respect to the adjoint action. It follows from \cref{prop:YDmodules} that $(\Uq(\g)^\lf, \tau)$ is a commutative algebra in $\ZDr(\Rep_q(G))$.

\begin{defn}
The \defterm{category of quantum Harish-Chandra bimodules} is
\[\HC_q(G) = \HC(\Rep_q(G), \Uq(\g)^\lf).\]
\end{defn}

\begin{remark}
A similar definition of the category of quantum Harish-Chandra bimodules is given in \cite[Definition 5.26]{VoigtYuncken}.
\end{remark}

\begin{remark}
By \cite[Theorem 3.10]{SafronovQMM} the notion of quantum moment maps in this setting coincides with the quantum moment maps for quantum group actions introduced in \cite[Section 1.5]{VaragnoloVasserot}.
\end{remark}

In this case the algebra $\cD$ from \cref{def:Heisenbergdouble} is the algebra of quantum differential operators $\cD_q(G)$ on $G$ (see \cite[Section 4.1]{BackelinKremnizer} where it is denoted by $\cD_q^{fin}$).

As in the case of classical Harish-Chandra bimodules, in the abelian case the category $\HC_q(G)$ has a straightforward description. Let $H$ be a torus and $\Lambda$ its weight lattice. Then $\Uq(\h)^\lf = \Uq(\h) = \cO(H)$ and $\HC_q(H)$ is equivalent to the category of $\Lambda$-graded $\cO(H)$-modules. There is a homomorphism
\[\Lambda\longrightarrow H\]
whose dual map $\cO(H)=k[\Lambda]\rightarrow \cO(\Lambda)$ on the level of functions is
\[K_\mu\mapsto \left(\lambda \mapsto q^{(\mu, \lambda)}\right)\]
for $\mu,\lambda\in\Lambda$. In particular, $\Lambda$ acts by translations on $H$ an we denote the induced functor by
\[(q^\lambda)^*\colon \LMod_{\cO(H)}\rightarrow \LMod_{\cO(H)}.\]
The monoidal structure $\otimes^{\HC}$ on $\HC_q(H)$ is given by
\[M\otimes^{\HC} N = \bigoplus_{\lambda\in\Lambda} (q^\lambda)^*(M)\otimes N(\lambda).\]

Suppose $V\in\Rep(H)$, $v\in V$ and $f\in\cO(H)\cong \Uq(\h)$. Then the map \eqref{eq:quantumfieldgoal} is
\[v\otimes f(\lambda)\mapsto f(\lambda q^{-h}) \otimes v\]
for $\lambda\in H$.

\subsection{Harish-Chandra bimodules and bialgebroids}

Let us again consider the general setup of \cref{sect:HCgeneral}, where $\cC$ is a cp-rigid monoidal category with a compact projective unit. In particular, $\HC$ is also a cp-rigid monoidal category with a compact projective unit. Our goal in this section is to describe a Tannaka reconstruction result for monoidal forgetful functors to $\HC$.

Recall from \cref{ex:convolutionproduct} that the category $\HC\otimes \HC$ carries two monoidal structures: the pointwise monoidal structure on $\HC\otimes \HC^{\mop}$ and the convolution product. We will call the latter the Takeuchi product in this setting.

\begin{defn}
The \defterm{Takeuchi product} $\times_\cL$ is the monoidal structure on $\HC\otimes \HC$ given by
\[(M_1\boxtimes M_2)\times_\cL (N_1\boxtimes N_2) = \ev(M_2, N_1)\otimes (M_1\boxtimes N_2)\]
with the unit $\coev(k)=\cD\in\HC\otimes \HC$.
\label{def:HCTakeuchiProduct}
\end{defn}

\begin{example}
Consider the setup of \cref{def:classicalHC}. An object of $\HC(G)\otimes \HC(G)\cong \HC(G\times G)$ is a $\U\g\otimes (\U\g)^{\op}$-bimodule with a certain integrability condition. For a $(\U\g)^{\op}$-bimodule $M$ and a $\U\g$-bimodule $N$ the Takeuchi product is the subspace
\[M\times_{\U\g} N\subset M\otimes_{\U\g} N\]
of elements $\sum_i m_i\otimes n_i$ satisfying
\[\sum_i m_i x\otimes n_i = m_i\otimes n_ix\]
for every $x\in\U\g$, see \cite{Takeuchi}.
\end{example}

We will now formulate the notion of bialgebroids in the category of Harish-Chandra bimodules. Recall that the algebra $\cD\cong T^\R(\cL)\in\cC\otimes \cC^{\op}$ carries a natural quantum moment map \eqref{eq:idmomentmap}.

\begin{defn}
A \defterm{Harish-Chandra bialgebroid} is an algebra $B\in\cC\otimes \cC^{\mop}$ equipped with a quantum moment map $s\otimes t\colon \cL\boxtimes \cL^{\op}\rightarrow B$, which allows us to regard $B$ as an algebra in $\HC\otimes \HC^{\mop}$, together with a coassociative coproduct $\Delta\colon B\rightarrow B\times_\cL B$, a map of algebras in $\HC\otimes \HC^{\mop}$, and a counit map $\varepsilon\colon B\rightarrow \cD$, a map of algebras in $\cC\otimes \cC^{\mop}$ compatible with quantum moment maps.
\label{def:HCbialgebroid}
\end{defn}

\begin{example}
Let $H$ be a split torus, $\Lambda=\Hom(H, \Gm)$ its weight lattice and consider the category of Harish-Chandra bimodules $\HC(H)$. A bialgebroid in $\HC(H)$ is given by the following data:
\begin{itemize}
\item An algebra with a bigrading
\[B = \bigoplus_{\alpha,\beta\in\Lambda} B_{\alpha\beta}.\]
\item Algebra maps
\[s, t\colon \cO(\h^*)\longrightarrow B\]
which satisfy the quantum moment map equations
\[s(f(\lambda)) a = a s(f(\lambda + \alpha)),\qquad t(f(\lambda)) a = a t(f(\lambda+\beta))\]
for $f\in\cO(\h^*)$ and $a\in B_{\alpha\beta}$.
\item The coproduct $\Delta\colon B\rightarrow B\times_{\U\h} B$, a map of algebras compatible with the grading and quantum moment maps. Here the Takeuchi product is
\[(B\times_{\U\h} B)_{\alpha\beta} = \bigoplus_{\delta\in\Lambda} B_{\alpha\delta}\otimes_{\cO(\h^*)} B_{\delta\beta},\]
where the relative tensor product is the quotient of the $k$-linear tensor product modulo the relations $t(f)a\otimes b\sim a\otimes s(f) b$ for $a\otimes b\in B_{\alpha\delta}\otimes B_{\delta\beta}$ and $f\in\cO(\h^*)$.
\item The counit $\epsilon\colon B\rightarrow \D(H)$, a map of algebras compatible with the grading and quantum moment maps.
\end{itemize}
We see that this data is essentially an $\h$-bialgebroid in the sense of \cite[Section 4.1]{EtingofVarchenkoSolutions} (note that the Mellin transform identifies $\D(H)$ with the algebra of $\Lambda$-difference operators on $\h^*$).
\label{ex:hbiaglebroids}
\end{example}

\begin{thm}
Suppose $B$ is a Harish-Chandra bialgebroid. The functor $\bot\colon \HC\rightarrow \HC$ given by
\[\bot(M) = B\times_\cL M\]
defines a lax monoidal comonad. Conversely, let $\bot\colon \HC\rightarrow \HC$ be a colimit-preserving lax monoidal comonad on $\HC$. Then $\bot(-)\cong B\times_\cL(-)$ for some Harish-Chandra bialgebroid $B$.
\label{thm:HClaxmonoidalcomonad}
\end{thm}
\begin{proof}
Recall from \cite[Definition 6.25]{AguiarMahajan} that a bimonoid in a duoidal category is an algebra with respect to one monoidal structure and a coalgebra with respect to the other monoidal structure, both compatible in a natural way. A coalgebra in $(\Fun^\L(\HC, \HC), \circ)$ is a colimit-preserving comonad on $\HC$ and a bimonoid in $\Fun^\L(\HC, \HC)$ is the same as lax monoidal comonad on $\HC$.

A colimit-preserving lax monoidal comonad on $\HC$ is the same as a bimonoid in the duoidal category $\Fun^\L(\HC, \HC)$. By \cref{thm:endofunctorsduoidal} we have an equivalence of duoidal categories $\HC\otimes \HC\cong \Fun^\L(\HC, \HC)$. So, $\bot$ corresponds to an object $B\in\HC\otimes \HC$ which is both an algebra in $\HC\otimes \HC^{\mop}$ as well as a coalgebra in $(\HC\otimes \HC, \times_\cL)$, both in a compatible way.

By \cref{lm:HCop} we have an equivalence of monoidal categories
\[\HC(\cC, \cL)\otimes \HC(\cC, \cL)^{\mop}\cong \HC(\cC, \cL)\otimes \HC(\cC^{\mop}, \cL^{\op}),\]
so by \cref{prop:HCalgebras} the data of an algebra $B\in\HC\otimes \HC^{\mop}$ boils down to an algebra $B\in\cC\otimes \cC^{\mop}$ equipped with a quantum moment map $\cL\boxtimes \cL^{\op}\rightarrow B$.

The data of a comonad boils down to a coalgebra $(B, \Delta, \epsilon)$ in $(\HC\otimes \HC, \times_\cL)$. The counit is given by a map of algebras $\epsilon\colon B\rightarrow \coev(k)$ in $\HC\otimes \HC^{\mop}$. Identifying algebras in $\HC\otimes \HC^{\mop}$ with algebras in $\cC\otimes \cC^{\mop}$ equipped with quantum moment maps by \cref{prop:HCalgebras}, the counit is the same as a map $\varepsilon\colon B\rightarrow T^\R(\cL)\cong \cD$ of algebras in $\cC\otimes \cC^{\mop}$ compatible with quantum moment maps from $\cL\boxtimes \cL^{\op}$.
\end{proof}

The definition of representations of Harish-Chandra bialgebroids is straightforward.

\begin{defn}
Suppose $B\in\HC\otimes \HC$ is a Harish-Chandra bialgebroid. A \defterm{$B$-comodule} is an object $M\in\HC$ together with a coassociative and counital coaction $M\rightarrow B\times_\cL M$.
\end{defn}

Equivalently, by \cref{thm:HClaxmonoidalcomonad} a $B$-comodule is a coalgebra over the comonad $\bot(M) = B\times_\cL M$.

\begin{example}
Consider the category of Harish-Chandra bimodules $\HC(H)$ for a split torus $H$ as in \cref{ex:hbiaglebroids} and let $B$ be a Harish-Chandra bialgebroid in $\HC(H)$. Then a $B$-comodule is a $\cO(\h^*)$-module
\[M = \bigoplus_{\alpha\in\Lambda} M_\alpha\]
together with a coaction map
\[M_\alpha\longrightarrow \bigoplus_{\beta\in\Lambda} B_{\alpha\beta}\otimes_{\cO(\h^*)} M_\beta,\]
where $B$ is considered as a right $\cO(\h^*)$-module via the \emph{left} action of $t\colon \cO(\h^*)\rightarrow B$. We require this coaction map to be compatible with the $\cO(\h^*)$-actions on both sides, where the $\cO(\h^*)$-action on the right is via the left multiplication by $s\colon \cO(\h^*)\rightarrow B$ and be coassociative and counital in the obvious way.
\end{example}

We obtain a Tannaka recognition statement for Harish-Chandra bialgebroids.

\begin{thm}
Suppose $\cD$ is a monoidal category with a monoidal functor $F\colon \cD\rightarrow \HC$ which admits a colimit-preserving right adjoint $F^\R\colon \HC\rightarrow \cD$. Then there is a Harish-Chandra bialgebroid $B$, such that $(F\circ F^\R)(-)\cong B\times_\cL(-)$ and $F$ factors through a monoidal functor
\[\cD\longrightarrow \CoMod_B(\HC).\]
If $F$ is conservative and preserves equalizers, the above functor is an equivalence.
\label{HarishChandraTannakaReconstruction}
\end{thm}
\begin{proof}
Since $F$ is monoidal, $F^\R$ is lax monoidal. Therefore, $\bot = FF^\R$ is a colimit-preserving lax monoidal comonad on $\HC$. By \cref{thm:HClaxmonoidalcomonad} there is a Harish-Chandra bialgebroid $B$, such that $\bot(-)\cong B\times_\cL(-)$. By the standard monadic arguments $F$ factors through
\[\cD\longrightarrow \CoMod_\bot(\HC)\cong \CoMod_B(\HC),\]
which is monoidal (see \cite[Proposition 3.5]{Szlachanyi} for the dual statement). If $F$ is conservative and preserves equalizers, by the Barr--Beck theorem \cite[Theorem VI.7.1]{Maclane} the above functor is an equivalence.
\end{proof}

\section{Dynamical \texorpdfstring{$R$}{R}-matrices}

In this section we explain how the dynamical twists and dynamical $R$-matrices arise from the categorical formalism explained in this paper.

\subsection{Dynamical twists}

Consider a Hopf algebra $H$, a commutative algebra $\cL\in\ZDr(\LMod_H)$ (see \cref{prop:YDmodules}) with a coaction map $\delta\colon \cL\rightarrow H\otimes \cL$ (denoted by $x\mapsto x_{(-1)}\otimes x_{(0)}$) and a cp-rigid monoidal category $\cC$ together with a forgetful functor $F\colon \cC\rightarrow \LMod_H$ which we assume sends compact projective objects in $\cC$ to finite-dimensional $H$-modules. It will be convenient to introduce the right $H$-coaction
\[\delta^R\colon \cL\longrightarrow \cL\otimes H\]
by
\[x\mapsto x_{(0)}\otimes S^{-1}(x_{(-1)}).\]

\begin{prop}
A monoidal structure with a strict unit map on the composite
\[\cC \longrightarrow \LMod_H \xrightarrow{\free} \HC(\LMod_H, \cL)\]
is the same as a natural collection of elements $J_{X, Y}\in \cL\otimes \Hom(F(X)\otimes F(Y), F(X\otimes Y))$ for $X,Y\in\cC^{\cp}$ satisfying
\begin{itemize}
\item The elements $J_{X, Y}$ are $H$-invariant.

\item For a triple $X,Y,Z\in\cC^{\cp}$ the equation
\[J_{X\otimes Y, Z}\circ(J_{X, Y}\otimes \id_Z) = J_{X, Y\otimes Z}\circ \delta^R_X(J_{Y, Z})\]
holds, where $\delta^R_X$ means the $H$-factor in $\delta^R$ acts on $X$.

\item For any $X\in\cC^{\cp}$ we have $J_{\bu, X}=J_{X, \bu} = 1\otimes \id_{F(X)}$.
\end{itemize}
\label{prop:dynamicaltwistmonoidal}
\end{prop}
\begin{proof}
Recall that the monoidal structure on the functor $\free\colon \LMod_H \rightarrow \HC(\LMod_H, \cL)$ is given by the natural isomorphism
\begin{align*}
(\cL\otimes X) \otimes_\cL (\cL\otimes Y) &\xrightarrow{\sim} \cL\otimes X \otimes Y\\
a\otimes x \otimes b \otimes y &\mapsto ab_{(0)} \otimes S^{-1}(b_{(-1)}) \triangleright x \otimes y
\end{align*}
for any $X,Y\in \LMod_H$. So, the monoidal structure on $\cC\rightarrow \HC(\LMod_H, \cL)$ is given by
\[(\cL\otimes F(X))\otimes_{\cL} (\cL\otimes F(Y))\cong \cL\otimes F(X)\otimes F(Y)\xrightarrow{\sim} \cL\otimes F(X\otimes Y),\]
where the first isomorphism is given by the monoidal structure on $\free$ and the second isomorphism is
\[l\otimes a\otimes b\mapsto (l\otimes \id_{F(X\otimes Y)})J_{X, Y}(a\otimes b).\]
The composite is automatically a map of $\cL$-modules and the compatibility with the $H$-action is the $H$-invariance condition on $J_{X, Y}$.

The associativity condition for the monoidal structure on $F$ is that for compact projective objects $X,Y,Z\in\cC^{\cp}$ the diagram
\[
\xymatrix@C=-0.7cm{
((\cL \otimes F(X)) \otimes_\cL (\cL \otimes F(Y))) \otimes_\cL (\cL \otimes F(Z)) \ar^{\sim}[rr] \ar^{J_{X,Y}\otimes \id}[d] & & (\cL \otimes F(X)) \otimes_\cL ((\cL \otimes F(Y)) \otimes_\cL (\cL \otimes F(Z))) \ar^{\id\otimes J_{Y, Z}}[d] \\
(\cL \otimes F(X\otimes Y)) \otimes_\cL (\cL\otimes F(Z)) \ar_{J_{X\otimes Y,Z}}[dr] & & (\cL\otimes F(X)) \otimes_\cL (\cL\otimes F(Y\otimes Z)) \ar^{J_{X, Y\otimes Z}}[dl] \\
& \cL\otimes F(X\otimes Y\otimes Z)
}
\]
commutes. Considering the image of $(1\otimes a)\otimes (1\otimes b)\otimes (1\otimes c)$ under these maps we get the second equation.

The unitality condition for the monoidal structure on $F$ is equivalent to the last equation.
\end{proof}

Let us now introduce a universal version of the previous statement. Suppose $A$ is another Hopf algebra with a map of algebras $H\rightarrow A$. We assume $\cC=\LMod_A$ and the forgetful functor
\[F\colon \LMod_A\longrightarrow \LMod_H\]
is given by restriction of modules. Denote by $(A, \Delta, \varepsilon)$ the coalgebra structure on $A$.

\begin{defn}
A \defterm{dynamical twist} is an invertible element
\[J=J^0 \otimes J^1 \otimes J^2 \in \cL \otimes A \otimes A\]
satisfying 
\begin{enumerate}
\item The invariance condition 
\[ h_{(1)} \triangleright J^0 \otimes h_{(2)} J^1 \otimes h_{(3)} J^2 = J^0\otimes J^1 h_{(1)} \otimes J^2  h_{(2)}\]
for every $h\in H$;
\item The shifted cocycle equation 
\[((\id\otimes \Delta \otimes \id)J) (J \otimes 1) = ((\id \otimes \id\otimes \Delta)J)(J_{(0)}^0 \otimes S^{-1}(J_{(-1)}^0) \otimes J^1 \otimes J^2);  \]
\item The normalization condition
\[ (\id \otimes \varepsilon \otimes \id) J = 1\otimes 1\otimes 1 =  (\id \otimes \id \otimes \varepsilon) J.\]
\end{enumerate}
\end{defn}

\begin{example}
Consider the trivial pair $H=\cL=k$ given by the ground field. The invariance condition is empty, while the cocycle equation and the normalization condition imply that $J\in A\otimes A$ is a (constant) twist for the Hopf algebra in the sense of \cite[Proposition 4.2.13]{ChariPressley}.
\end{example}

\begin{example}
Suppose $\h$ is an abelian Lie algebras and consider $H=\cL=\U\h$ as in \cref{sect:classicalHC}. Then a dynamical twist is a function $J\colon \h^*\rightarrow A\otimes A$. The invariance condition is that $J(\lambda)$ is $\h$-invariant with respect to the adjoint action (the \emph{zero-weight condition}). The shifted cocycle equation is
\[((\Delta\otimes\id) J(\lambda))J_{12}(\lambda) = ((\id\otimes\Delta) J(\lambda)) J_{23}(\lambda-h^{(1)}).\]
\label{ex:additivitedynamicaltwist}
\end{example}

\begin{prop}
The data of a dynamical twist is equivalent to the data of a monoidal structure on $\LMod_A \rightarrow \HC(\LMod_H, \cL)$ with a strict unit map.
\label{prop:universaldynamicaltwistmonoidal}
\end{prop}
\begin{proof}
By \cref{prop:dynamicaltwistmonoidal} the monoidal structure is specified by a collection of elements 
\[J_{X, Y}\in \cL\otimes \End(X\otimes Y),\qquad X,Y\in\LMod_A.\]
By naturality these are uniquely determined by the elements
\[J = J_{A, A}(1_A\otimes 1_A)\in\cL\otimes A\otimes A.\]
\end{proof}

Two dynamical twists may be related by a gauge transformation.

\begin{defn}
A \defterm{gauge transformation} is an invertible $H$-invariant element $G\in\cL \otimes A$ satisfying the normalization condition
\[ (\id \otimes \varepsilon) (G) = 1 \otimes 1.  \]
\end{defn}

Given a dynamical twist $J$ and a gauge transformation $G$ we obtain a new dynamical twist by the formula
\begin{equation}
J^G = (\id \otimes \Delta)G\cdot J\cdot ((\delta^R\otimes\id)G)^{-1}(G \otimes 1)^{-1}.
\label{eq:dynamicalgaugetransformation}
\end{equation}

\begin{example}
Consider the pair $H=\cL=\U\h$ as in \cref{ex:additivitedynamicaltwist}. Then a gauge transformation is a zero-weight function $G\colon \h^*\rightarrow A^\times$ satisfying $\varepsilon(G(\lambda)) = 1$. Given a dynamical twist $J\colon \h^*\rightarrow A\otimes A$, its gauge transformation is
\[J^G(\lambda) = (\id\otimes \Delta)G(\lambda) \cdot J(\lambda) \cdot (G_2(\lambda-h^{(1)}))^{-1}(G_1(\lambda))^{-1}.\]
\end{example}

\begin{prop}
Suppose $J_1, J_2$ are two dynamical twists which give rise to monoidal structures on the functor $\LMod_A\rightarrow \HC(\LMod_H, \cL)$ by \cref{prop:universaldynamicaltwistmonoidal}. The data of a gauge transformation between them is a monoidal natural isomorphism
\[
\xymatrix@C=3cm{
\LMod_A \rtwocell^{J_1}_{J_2} & \HC(\LMod_H, \cL)
}
\]
\end{prop}

\subsection{Dynamical FRT and reflection equation algebras}

Let us describe the Harish-Chandra bialgebroid $B$ from \cref{HarishChandraTannakaReconstruction} explicitly. Let $\cD$ be a cp-rigid monoidal category and $F\colon \cD\rightarrow \HC$ a monoidal functor which admits a colimit-preserving right adjoint $F^\R$.

By \cref{prop:rightadjoint} the functor $FF^\R$ can be calculated as
\begin{align*}
FF^\R(x) &= \int^{y\in \cD^\cp} \Hom_{\HC}(F(y)^\vee, x) \otimes F(y)^\vee \\
&\cong \int^{y\in\cD^\cp} \Hom_{\HC}(\cL, F(y)\otimes_\cL x)\otimes F(y)^\vee.
\end{align*}
Recalling the definition of the Takeuchi product from \cref{def:HCTakeuchiProduct}, we obtain that $FF^\R(x)\cong B\times_\cL x$, where the Harish-Chandra bialgebroid $B$ is
\[B\cong \int^{y\in\cD^\cp} F(y)^\vee\boxtimes F(y)\in\HC\otimes\HC.\]

As in \cref{sect:TannakaBialgebras} denote by
\[\pi_y\colon F(y)^\vee\boxtimes F(y)\longrightarrow B\]
the natural projections. The Harish-Chandra bialgebroid structure is given on generators as follows:
\begin{enumerate}
\item The coproduct
\[B\longrightarrow B\times_{\cL} B\cong \int^{(y,z)\in \cD^\cp \times \cD^\cp} \Hom_{\HC}(\cL, F(y)\otimes_\cL F(z)^\vee)\otimes F(y)^\vee\boxtimes F(z)\]
is
\[F(y)^\vee\boxtimes F(y)\xrightarrow{\coev\otimes \id} \Hom_{\HC}(\cL, F(y)\otimes_\cL F(y)^\vee)\otimes F(y)^\vee\boxtimes F(y)\xrightarrow{\pi_{y, y}} B\times_\cL B.\]

\item The counit
\[B\longrightarrow T^\R_{\HC}(\cL) = \int^{P\in\HC^{\cp}} P^\vee\boxtimes P\]
is the projection
\[F(y)^\vee\boxtimes F(y) \rightarrow T^\R_{\HC}(\cL).\]

\item The product is the composite
\begin{align*}
(F(y)^\vee \boxtimes F(y)) \otimes_{\cL\otimes \cL} (F(z)^\vee \boxtimes F(z)) &= (F(y)^\vee\otimes_\cL  F(z)^\vee)\boxtimes (F(z)\otimes_\cL F(y)) \\
&\cong (F(z)\otimes_\cL F(y))^\vee\boxtimes (F(z)\otimes_\cL F(y)) \\
&\xrightarrow{(J^{-1}_{z,y})^\vee\boxtimes J_{z,y}} F(z\otimes y)^\vee\boxtimes F(z\otimes y) \\
&\xrightarrow{\pi_{z\otimes y}} B.
\end{align*}

\item The quantum moment map is 
\[ \cL \boxtimes \cL^{\op} \cong F(\bu) \boxtimes F(\bu) \xrightarrow{\pi_{\bu}} B. \]
\end{enumerate}

We will now concentrate on the case $\cC=\Rep(H)$ for $H$ a split torus with weight lattice $\Lambda$ and $\cL=\U\h$, so that $\HC=\HC(H)$. Moreover, we assume that the functor $F\colon \cD\rightarrow \HC(H)$ factors as the composite
\[\cD\longrightarrow \Rep(H)\xrightarrow{\free} \HC(H).\]
For an object $y\in \cD$ we denote its image in $\Rep(H)$ by the same letter. In this case the monoidal structure is given by a dynamical twist
\[J_{y, z}(\lambda)\colon \h^*\rightarrow \End(y\otimes z)\]
as in \cref{prop:dynamicaltwistmonoidal}.

The projections
\[\pi_y\colon (\U\h\otimes y^\vee)\boxtimes (\U\h\otimes y)\longrightarrow B\]
in $\HC(H\times H)$ may be encoded in elements
\[T_y\in B\otimes \End(y).\]

Analogously to \cref{thm:Bmatrixpresentation} we obtain the following explicit description of the bialgebroid $B$.

\begin{thm}
The bialgebroid $B$ is spanned, as an $\cO(\h^*)$-bimodule, by the matrix coefficients of $T_y$ for $y\in\cD^{\cp}$ subject to the relation
\[F(j)\circ T_x = T_y\circ F(j)\]
for every $j\colon x\rightarrow y$. Moreover, we have:
\begin{enumerate}
\item $\Delta(T_y) = T_y\otimes T_y$ for every $y\in\cD^{cp}$.

\item $\epsilon(T_y) = 1\otimes\id_y\in \D(H)\otimes \End(y)$ for every $y\in\cD^{\cp}$.

\item For every $f\in\cO(\h^*)$ and $y\in\cD^{\cp}$
\begin{align*}
s(f(\lambda)) T_y &= T_y s(f(\lambda + h)) \\
t(f(\lambda+h)) T_y &= T_y t(f(\lambda )).
\end{align*}

\item $J^t_{y, z}(\lambda)^{-1} T_{y\otimes z} J^s_{y, z}(\lambda) = (T_y\otimes \id)(\id\otimes T_z)$, where by the superscripts we mean the \emph{left} multiplication with the $\cO(\h^*)$-part by either the source $(s)$ or the target $(t)$ map.

\item $T_\bu\in B$ is the unit.
\end{enumerate}
\label{thm:Tpresentation}
\end{thm}

\begin{defn}
Suppose $\cD$ is a braided monoidal category together with a forgetful functor $\cD\rightarrow \Rep(H)$ and a monoidal structure on the composite $\cD\rightarrow \Rep(H)\xrightarrow{\free} \HC(H)$. For $x,y\in\cD$ define the morphism $\U\h\otimes x\otimes y\rightarrow \U\h\otimes y\otimes x$ by
\[\check{R}_{x,y} \colon F(x) \otimes_{\U\h} F(y) \xrightarrow{J_{x,y}} F(x\otimes y) \xrightarrow{F(\sigma_{x,y})} F(y\otimes x) \xrightarrow{J_{x,y}^{-1}} F(y) \otimes_{\U\h} F(x).\]
The \defterm{dynamical $R$-matrix} is the map $R_{x, y}\colon \h^*\rightarrow \End(x\otimes y)$ given by
\[R_{x, y} = (\id_{\U\h}\otimes \sigma_{x, y}^{-1})\circ \check{R}_{x, y}.\]
\end{defn}

As in \cref{sect:TannakaBialgebras}, we use the standard notation $T_1 = T_x \otimes \id$ and $T_2=\id\otimes T_y$ and similarly for the $R$-matrix.

\begin{prop}
Let $x,y,z\in\cD^{\cp}$.
\begin{enumerate}
\item The dynamical $R$-matrix satisfies the dynamical Yang-Baxter equation
\[ R_{23}(\lambda)R_{13}(\lambda-h^{(2)}) R_{12}(\lambda) = R_{12}(\lambda-h^{(3)}) R_{13}(\lambda) R_{23}(\lambda-h^{(1)})\]
in $\End(x\otimes y\otimes z)$.
\item The element $T$ satisfies the dynamical FRT relation
\[ R^t(\lambda) T_1 T_2 = T_2 T_1 R^s(\lambda)\]
in $B\otimes \End(x\otimes y)$.
\end{enumerate}
\label{prop:dynamicalFRT}
\end{prop}
\begin{proof}
As in the proof of \cref{prop:YBRTTequations}, the element $\hat{R}$ satisfies the braid relation
\[\check{R}_{12} \check{R}_{23}  \check{R}_{12} = \check{R}_{23} \check{R}_{12} \check{R}_{23}\]
in $\U\h\otimes\End(x\otimes y\otimes z)$. Observing that $\check{R} = \sigma \circ R$, we get the dynamical Yang--Baxter equation.

To show the second part, recall from \cref{thm:Tpresentation} that
\[F(\sigma_{x,y}) T_{x\otimes y} = T_{y\otimes x} F(\sigma_{x,y}).\]
Decomposing $T_{x\otimes y}$ and $T_{y\otimes x}$ into $T_x$ and $T_y$ using property $(4)$ of the same theorem, we get the result.
\end{proof}

\section{Fusion of Verma modules}

\label{sect:fusion}

In this section we construct standard dynamical twists for $\U\g$ and $\U_q(\g)$ using the so-called exchange construction introduced in \cite{EtingofVarchenkoExchange}.

\subsection{Classical parabolic restriction}
\label{sect:classicalres}

Let $G$ be a reductive group over an algebraically closed field $k$ of characteristic zero and $\g$ its Lie algebra. Fix a Borel subgroup $B\subset G$ and denote $H = B / [B, B]$; their Lie algebras are denoted by $\b$ and $\h$. We denote by $N$ the kernel of $B\rightarrow H$ with Lie algebra $\n$. Let $W$ be the Weyl group.

We will later use the Harish-Chandra isomorphism, see \cite[Theorem 1.10]{HumphreysBGG}.

\begin{thm}
There is a unique homomorphism of algebras
\[\hc\colon \rZ(\U\g)\longrightarrow \U\h,\]
the \defterm{Harish-Chandra homomorphism}, such that for any $z\in\rZ(\U\g)$ and $m\in M^{\univ}$ we have
\[z m = m \hc(z).\]
\label{thm:HChomomorphism}
\end{thm}

\begin{defn}
The \defterm{universal category $\cO$} is the category $\cO^{\univ}$ of $(\U\g, \U\h)$-bimodules whose diagonal $\b$-action integrates to a $B$-action. The \defterm{universal Verma module} is
\[M^{\univ} = \U\g\otimes_{\U\b} \U\h \in \cO^{\univ}.\]
\end{defn}

\begin{remark}
Just like the usual category $\cO$ is constructed to contain objects like Verma modules, we define $\cO^{\univ}$ to contain objects like universal Verma modules.
\end{remark}

\begin{remark}
We may identify $\cO^{\univ}$ with the category of $\U\g$-modules in the category $\Rep(H)$ whose $\n$-action is locally nilpotent.
\label{rmk:Ounivhighestweight}
\end{remark}

We will now define an important bimodule structure on $\cO^{\univ}$:
\begin{equation}
\HC(G)\curvearrowright \cO^{\univ} \curvearrowleft \HC(H).
\label{eq:classicalcorrespondence}
\end{equation}
Both actions are given by the relative tensor products of bimodules. Given a $\U\g$-bimodule $X\in\HC(G)$ and a $(\U\g, \U\h)$-bimodule $M\in\cO^{\univ}$, $X\otimes_{\U\g} M$ is an $(\U\g, \U\h)$-bimodule. Since the diagonal $\g$-action on $X$ is integrable, so is the diagonal $\b$-action. Therefore, the diagonal $\b$-action on $X\otimes_{\U\g} M$ is integrable. The $\HC(H)$ action is defined similarly.

Let
\[\act_G\colon \HC(G)\longrightarrow \cO^{\univ},\qquad \act_H\colon \HC(H)\longrightarrow \cO^{\univ}\]
be the actions of $\HC(G)$ and $\HC(H)$ on the universal Verma module $M^{\univ}\in \cO^{\univ}$. Using \cref{prop:laxmonoidalactionadjoint} we obtain the following lax monoidal functors.

\begin{defn}
The \defterm{parabolic restriction} is the lax monoidal functor
\[\res= \act^\R_H\circ \act_G\colon \HC(G)\longrightarrow \HC(H).\]
The \defterm{parabolic induction} is the lax monoidal functor
\[\ind = \act^\R_G\circ \act_H\colon \HC(H)\longrightarrow \HC(G).\]
\end{defn}

Let us now make these functors more explicit. Consider the functor
\[(-)^N\colon \cO^{\univ}\longrightarrow \HC(H)\]
which sends a $(\U\g, \U\h)$-bimodule to the subspace of highest weight vectors with respect to the $\U\g$-action. It still has a remaining $\U\h$-bimodule structure and so it defines an object of $\HC(H)$.

\begin{prop}
The functor $(-)^N\colon \cO^{\univ}\rightarrow \HC(H)$ is right adjoint to $\act_H\colon \HC(H)\rightarrow \cO^{\univ}$.
\label{prop:resexplicit}
\end{prop}
\begin{proof}
Identify $\cO^{\univ}$ with highest-weight $\U\g$-modules in the category $\Rep(H)$ following \cref{rmk:Ounivhighestweight}. For $M\in\cO^{\univ}$ and $X\in\HC(H)$ we have
\begin{align*}
\Hom_{\cO^{\univ}}(\act_H(X), M) &= \Hom_{\cO^{\univ}}(\U\g\otimes_{\U\b} X, M) \\
&\cong \Hom_{\BiMod{\U\b}{\U\h}}(X, M) \\
&\cong \Hom_{\HC(H)}(X, M^N).
\end{align*}
\end{proof}

So,
\[\res(X)\cong (X/X\n)^N.\]
The lax monoidal structure on $\res$ can be described explicitly as follows. For $X,Y\in \HC(G)$ the morphism
\begin{equation}
(X/X\n)^N \otimes_{\U\h} (Y/Y/\n)^N \longrightarrow (X\otimes_{\U\g} Y/(X\otimes_{\U\g} Y)\n)^N
\label{eq:resmonoidal}
\end{equation}
is given by $[x] \otimes [y] \mapsto [x \otimes y]$. This assignment is independent of the choice of a representative of $[x]$ since $[y]$ is $N$-invariant.

\begin{remark}
Since $\res\colon \HC(G)\rightarrow \HC(H)$ is lax monoidal, it sends algebras in $\HC(G)$ to algebras in $\HC(H)$. By \cref{prop:HCalgebras}, an algebra in $\HC(G)$ is a $G$-algebra equipped with a quantum moment map $\mu\colon \U\g\rightarrow A$. It is easy to see that $\res(A)$ is the quantum Hamiltonian reduction $A\ham N$. This algebra is known as the Mickelsson algebra \cite{Mickelsson}, we refer to \cite{ZhelobenkoSalgebras} for more details.
\label{rmk:Mickelsson}
\end{remark}

Recall that the coinduction functor
\[\coind_B^G\colon \Rep(B)\longrightarrow \Rep(G)\]
is right adjoint to the obvious restriction functor $\Rep(G)\rightarrow \Rep(B)$. Denote in the same way the functor
\[\coind_B^G\colon \cO^{\univ}\longrightarrow \HC(G)\]
of coinduction from $B$ to $G$ using the diagonal $B$-action.

\begin{prop}
The functor $\coind_B^G\colon \cO^{\univ}\rightarrow \HC(G)$ is right adjoint to $\act_G\colon \HC(G)\rightarrow \cO^{\univ}$.
\label{prop:indexplicit}
\end{prop}
\begin{proof}
For $M\in\cO^{\univ}$ and $X\in\HC(G)$ we have
\begin{align*}
\Hom_{\cO^{\univ}}(\act_G(X), M) &= \Hom_{\cO^{\univ}}(X\otimes_{\U\b} \U\h, M) \\
&\cong \Hom_{\BiMod{\U\g}{\U\b}}(X, M).
\end{align*}
Both $X$ and $M$ are $(\U\g, \U\b)$-bimodules whose diagonal $\b$-action integrates to a $B$-action, i.e. they are objects of $\LMod_{\U\g}(\Rep B)$. Moreover, $X$ lies in the image of the forgetful functor
\[\HC(G)=\LMod_{\U\g}(\Rep G)\rightarrow \LMod_{\U\g}(\Rep B).\]
But by definition $\coind_B^G$ is the right adjoint to the forgetful functor $\Rep G\rightarrow \Rep B$.
\end{proof}

Let us now compute the values of $\res$ and $\ind$ on the units.

\begin{prop}
The natural morphism $\U\h\rightarrow \res(\U\g)$ is an isomorphism.
\label{prop:classicalresunit}
\end{prop}
\begin{proof}
By \cref{prop:resexplicit} $\res(\U\g) \cong (M^{\univ})^N$ and we have to show that
\[\U\h\longrightarrow (M^{\univ})^N\]
is an isomorphism. Let
\[M^{\univ, \gen} = \U\g\otimes_{\U\b}\Frac(\U\h),\]
where $\Frac(\U\h)$ is the fraction field of $\U\h$. The map $M^{\univ}\rightarrow M^{\univ, \gen}$ is injective and $(-)^N$ is left exact, so $(M^{\univ})^N\longrightarrow (M^{\univ, \gen})^N$ is injective. But the Verma module for generic highest weights is irreducible (see \cite[Theorem 4.4]{HumphreysBGG}), so
\[\Frac(\U\h)\longrightarrow (M^{\univ, \gen})^N\]
is an isomorphism. This implies the claim.
\end{proof}

\begin{cor}
The induced map
\[\res\colon \rZ(\U\g)=\End_{\HC(G)}(\U\g)\longrightarrow \U\h=\End_{\HC(H)}(\U\h)\]
coincides with the Harish-Chandra homomorphism $\hc\colon \rZ(\U\g)\rightarrow \U\h$.
\end{cor}
\begin{proof}
The map
\[\act_G\colon \rZ(\U\g)=\End_{\HC(G)}(\U\g)\longrightarrow \End_{\cO^{\univ}}(M^{\univ})\]
sends a central element $z\in\rZ(\U\g)$ to the left action of $z\in\rZ(\U\g)$ on $M^{\univ}$. By \cref{thm:HChomomorphism} it is equal to the right action of $\hc(z)\in\U\h$ on $M^{\univ}$. To conclude, observe that the map
\[\U\h\longrightarrow (M^{\univ})^N\]
is an isomorphism of right $\U\h$-modules.
\end{proof}

\begin{prop}
Suppose $G$ is connected and simply-connected. Then there is an isomorphism
\[\ind(\U\h)\cong \U\g\otimes_{\rZ(\U\g)} \U\h,\]
where the $\rZ(\U\g)$-action on $\U\h$ is via the Harish-Chandra homomorphism $\hc$.
\end{prop}
\begin{proof}
By \cref{prop:indexplicit} $\ind(\U\h)\cong \coind_B^G(M^{\univ})$. Identifying $B$-representations with $G$-equivariant quasi-coherent sheaves on $G/B$, $M^{\univ}$ is sent to $(\pi_* \D_{G/N})^H$, where $\pi\colon G/N\rightarrow G/B$. Therefore,
\[\coind_B^G(M^{\univ})\cong \D(G/N)^H.\]
The claim then follows from \cite{Shapovalov,HomberghdeVries}, see also \cite[Lemma 3.1]{Milicic}.
\end{proof}

Note that the functor $\res$ preserves neither limits nor colimits and it is merely lax monoidal. We will now show that after a localization it becomes exact and monoidal.

\begin{defn}
A weight $\lambda\in\h^*$ is \defterm{generic} if $\langle \lambda, \alpha^\vee\rangle\not\in \bZ$ for every root $\alpha$. Denote by $\h^{*, \gen}\subset \h^*$ the subset of generic weights. Let $\HC(H)^{\gen}\subset \HC(H)$ and $\cO^{\univ, \gen}\subset \cO^{\univ}$ be the full subcategories of right $\U\h$-modules supported on generic weights. Let $(\U\h)^{\gen}\subset \Frac(\U\h)$ be the subspace of rational functions on $\h^*$ regular on $\h^{*, \gen}$.
\label{def:genericweight}
\end{defn}

By construction
\[\HC(H)^{\gen} = \HC(\Rep(H), (\U\h)^{\gen})\]
and similarly for $\cO^{\univ, \gen}$. Moreover, both $\HC(H)^{\gen}\subset \HC(H)$ and $\cO^{\univ, \gen}\subset \cO^{\univ}$ admit left adjoints given by localization. Let
\[M^{\univ, \gen} = \U\g\otimes_{\U\b} (\U\h)^{\gen}\]
be the universal Verma module with generic highest weights.

Choose a Borel subgroup $B_-\subset G$ opposite to $B$, with Lie algebra $\b_-$. Let
\[M^{\univ}_- = \U\h\otimes_{\U\b_-} \U\g\]
be the opposite universal Verma module.

\begin{defn}
The functor of \defterm{$\n_-$-coinvariants}
\[(-)_{\n_-}\colon \cO^{\univ}\longrightarrow \HC(H)\]
is $M_{\n_-} = M^{\univ}_-\otimes_{\U\g} M$.
\end{defn}

We will now recall the \defterm{extremal projector} introduced in \cite{AsherovaSmirnovTolstoy}, see also \cite{ZhelobenkoSalgebras}.

\begin{thm}
There is an extension $T(\g)$ of $\U\g$ obtained by replacing $\U\h\subset\U\g$ with $\Frac(\U\h)$ and considering certain power series. There is an element $P\in T(\g)$ satisfying the following properties:
\begin{enumerate}
\item $\n P = P\n_- = 0$.

\item $P-1\in T(\g)\n\cap \n_- T(\g)$.
\end{enumerate}
The action of $P$ is well-defined on left $\U\g$-modules whose $\n$-action is locally nilpotent and which have generic $\h$-weights.
\label{thm:extremalprojector}
\end{thm}

\begin{example}
Suppose $\g=\sl_2$. The extremal projector in this case is (see e.g. \cite{KhoroshkinOgievetsky})
\[P = \sum_{n=0}^\infty \frac{(-1)^n}{n!} g_n^{-1} f^n e^n,\]
where
\[g_n = \prod_{j=1}^n (h+j+1).\]
\end{example}

We will now describe some applications of extremal projectors.

\begin{prop}
There is a natural isomorphism of functors $(-)_{\n_-}\cong (-)^N\colon \cO^{\univ,\gen}\rightarrow \HC(H)^{\gen}$. In particular, they are exact.
\label{prop:coinvinv}
\end{prop}
\begin{proof}
Take $M\in\cO^{\univ, \gen}$ and consider the composite
\[\pi\colon M^N\longrightarrow M\longrightarrow M_{\n_-}.\]
We will prove that it is an isomorphism.

Since the weights of the right $\U\h$-action on $M$ are generic and the weights of the diagonal $\U\h$-action are integral, the weights of the left $\U\h$-action are also generic. Moreover, the left $\U\n$-action is locally nilpotent. In particular, the action of the extremal projector from \cref{thm:extremalprojector}
\[P\colon M\longrightarrow M\]
is well-defined. It lands in $N$-invariants by the property $\n P = 0$. It factors through $\n_-$-coinvariants by the property $P\n_- = 0$. So, it gives a map
\[P\colon M_{\n_-}\longrightarrow M^N.\]
For $m\in M^N$ we have $Pm = m$ since $P-1\in T(\g)\n$. In particular, $P\circ \pi = \id$. For $m\in M$, $[m] = [Pm]$ in $M_{\n_-}$ since $P-1\in \n_- T(\g)$. In particular, $\pi\circ P = \id$.
\end{proof}

\begin{thm}
The category $\cO^{\univ, \gen}$ is free of rank 1 as a $\HC(H)^{\gen}$-module category in the sense of \cref{def:freerank1}.
\label{thm:Ofreerank1}
\end{thm}
\begin{proof}
The unit of the adjunction $\act_H\dashv (-)^N$ between $\HC(H)^{\gen}$ and $\cO^{\univ, \gen}$ is
\[X\longrightarrow (\act_H(X))^N\xrightarrow{\sim} (M^{\univ}\otimes_{\U\h} X)_{\n_-}.\]
By the PBW isomorphism this map is an isomorphism. In particular, $\act_H\colon \HC(H)^{\gen}\rightarrow \cO^{\univ, \gen}$ is fully faithful.

Since the $\n$-action on $M\in\cO^{\univ, \gen}$ is locally nilpotent, $M^N = 0$ if, and only if, $M=0$. But $(-)^N\colon \cO^{\univ, \gen}\rightarrow \HC(H)^{\gen}$ is exact by \cref{prop:coinvinv}. Therefore, it is conservative. Since its left adjoint $\act_H$ is fully faithful, it is an equivalence.
\end{proof}

\begin{cor}
The composite
\[\res^{\gen}\colon \HC(G)\xrightarrow{\res} \HC(H)\longrightarrow \HC(H)^{\gen}\]
is strongly monoidal and colimit-preserving.
\label{cor:resmonoidal}
\end{cor}
\begin{proof}
By \cref{thm:Ofreerank1} $\cO^{\univ, \gen}$ is free of rank 1 as a $\HC(H)^{\gen}$-module category. The claim then follows from \cref{prop:freerank1monoidal}.
\end{proof}

\begin{remark}
Consider the morphism of stacks $p\colon [\b/B]\rightarrow [\h/H]$. It admits a section $s\colon [\h/H]\rightarrow [\b/B]$, so that $p\circ s=\id_{[\h/H]}$. It is shown in \cite{Sevostyanov} that, restricting to generic weights, there is a homotopy $s\circ p\sim \id_{[\b/B]}$ given by the classical limit of the extremal projector. The proof of \cref{thm:Ofreerank1} gives an analogous intepretation of the extremal projector on the quantum level.
\end{remark}

We will now show that $\res^{\gen}$ gives rise to a dynamical twist. For this, according to \cref{prop:dynamicaltwistmonoidal}, we have to show that $\res^{\gen}$ of a free Harish--Chandra bimodule is free, i.e. we have to establish an isomorphism between $(V\otimes M^\univ)^N$ and $V\otimes (\U\h)^{\gen}$ in $\HC(H)^{\gen}$, for every $V\in\Rep(G)$.

\begin{thm}
The morphism
\[(V\otimes M^{\univ, \gen})^N\subset V\otimes M^{\univ, \gen} \longrightarrow V\otimes (\U\h)^{\gen},\]
where the second morphism is induced by the projection $M^{\univ, \gen}\rightarrow (\U\h)^{\gen}$ onto highest weights, defines a natural isomorphism witnessing commutativity of the diagram
\[
\xymatrix{
\Rep(G) \ar[d] \ar^{\free_G}[r] & \HC(G) \ar^{\res^{\gen}}[d] \\
\Rep(H) \ar^{\free_H}[r] & \HC(H)^{\gen}
}
\]
\label{thm:EVclassicalbasis}
\end{thm}
\begin{proof}
Let $M_\lambda$ be the Verma module of a generic highest weight $\lambda\in\h^*$ and denote by $x_\lambda\in M_\lambda$ the highest weight vector. We have to show that the map $(V\otimes M_\lambda)^N\rightarrow V$ given by
\[v\otimes x_\lambda + \dots\mapsto v\otimes 1,\]
where $\dots$ contain elements of $M_\lambda$ of weight less than $\lambda$, is an isomorphism. This is the content of \cite[Theorem 8]{EtingofVarchenkoExchange}.
\end{proof}

\begin{remark}
The $(\U\h)^{\gen}$-module $\res^{\gen}(V\otimes \U\g)$ admits \emph{another} natural basis constructed in \cite{Khoroshkin}.
\end{remark}

Consider $V,W\in\Rep(G)$. Let us recall that Etingof and Varchenko \cite{EtingofVarchenkoExchange} have introduced the \defterm{fusion matrix}
\[J^{EV}_{V, W}(\lambda)\colon V\otimes W\rightarrow V\otimes W\]
depending rationally on a parameter $\lambda\in\h^*$ as follows. Consider the Verma module $M_\lambda$ with highest weight $\lambda\in\h^*$. For $V\in\Rep(G)$ denote by $V=\oplus_{\lambda\in\Lambda} V[\lambda]$ its weight decomposition. Consider a morphism $M_\lambda\rightarrow M_\mu\otimes V$. The image of a highest-weight vector $x_\lambda\in M_\lambda$ has the form
\[x_\mu\otimes v + \dots,\]
where $\dots$ denote terms containing elements of $M_\mu$ of lower weight. This determines a morphism
\begin{equation}
\Hom_{\U\g}(M_\lambda, M_\mu\otimes V)\longrightarrow V[\lambda-\mu].
\label{verma_iso}
\end{equation}
For generic $\mu$ it is an isomorphism and for $v\in V[\lambda-\mu]$ we denote by $\Phi^v_\lambda\in \Hom_{\U\g}(M_\lambda, M_\mu\otimes V)$ the preimage of $v$ under this map.

For $v\in V$ and $w\in W$ of weights $\wt(v)$ and $\wt(w)$ consider the composite
\begin{equation}
M_\lambda\xrightarrow{\Phi_\lambda^v} M_{\lambda-\wt(v)}\otimes V\xrightarrow{\Phi^w_{\lambda-\wt(v)}\otimes \id} M_{\lambda-\wt(v)-\wt(w)} \otimes W\otimes V.
\label{eq:classicalfusioncomposite}
\end{equation}
The fusion matrix is defined so that this composite is $\Phi^{J^{EV}_{W, V}(\lambda)(w\otimes v)}_\lambda$. By \cite[Theorem 48]{EtingofVarchenkoExchange} $J^{EV}_{W, V}(\lambda)$ quantizes the standard rational solution of the dynamical Yang--Baxter equation (see \cite[Theorem 3.2]{EtingofVarchenkoGeometry}).

Combining \cref{cor:resmonoidal} and \cref{thm:EVclassicalbasis} we obtain a monoidal structure on the composite
\[\Rep(G)\longrightarrow \Rep(H)\longrightarrow \HC(H)^{\gen}.\]
In particular, as in \cref{prop:dynamicaltwistmonoidal} this gives rise to linear maps
\[J_{V, W}(\lambda)\colon V\otimes W\rightarrow V\otimes W\]
depending rationally on $\lambda\in\h^*$.

\begin{prop}
Let $V,W\in\Rep(G)$. The map $J_{V, W}(\lambda)\colon V\otimes W\rightarrow V\otimes W$ coincides with a permutation of the fusion matrix:
\[J_{V, W}(\lambda) = \tau J^{EV}_{W,V}(\lambda)\tau,\]
where $\tau$ is the flip of tensor factors.
\label{prop:classicalfusion}
\end{prop}
\begin{proof}
Let $x^{\univ}\in M^{\univ}$ be the generator of the universal Verma module and $x_\lambda\in M_\lambda$ be the generator of the Verma module of highest weight $\lambda$. Using the PBW identification $M^{\univ}\cong \U\n_-\otimes \U\h$ we identify elements of $M^{\univ}$ with functions $\h^*\rightarrow \U\n_-$.

For $v\in V$ we denote by $\sum v_i\otimes a_i x^{\univ}$ the unique highest-weight element of $V\otimes M^{\univ}$ which has an expansion $v\otimes x^{\univ}+\dots$. Similarly, for $w\in W$ we denote by $\sum w_i\otimes b_i x^{\univ} = w\otimes x^{\univ} + \dots$ the highest-weight element of $W\otimes M^{\univ}$.

Under the morphism \eqref{eq:resmonoidal}
\[(V\otimes M^{\univ})^N\otimes_{\U\h} (W\otimes M^{\univ})^N\longrightarrow (V\otimes W\otimes M^{\univ})^N\]
we have
\[\sum_{i, j} (v_i\otimes a_i x^{\univ})\otimes (w_j\otimes b_j x^{\univ})\mapsto \sum_{i, j} v_i\otimes (a_i)_{(1)} w_j\otimes (a_i)_{(2)} b_j x^{\univ}.\]

It is then easy to see that
\[J_{V,W}(\lambda)(v\otimes w) = \sum_i v_i\otimes a_i(\lambda-\wt(v)) w.\]

Using the same notations, the map $\Phi_\lambda^v\colon M_\lambda\rightarrow M_{\lambda-\wt(v)} \otimes V$ is
\[x_\lambda\mapsto \sum_i a_i(\lambda-\wt(v)) x_{\lambda-\wt(v)}\otimes v_i.\]
Therefore, the composite \eqref{eq:classicalfusioncomposite} is
\begin{align*}
x_\lambda &\mapsto \sum_i a_i(\lambda-\wt(v)) x_{\lambda-\wt(v)}\otimes v_i\\
&\mapsto \sum_{i,j} a_i(\lambda-\wt(v))_{(1)} b_j(\lambda-\wt(v)-\wt(w))x_{\lambda-\wt(v)-\wt(w)}\otimes a_i(\lambda-\wt(v))_{(2)} w_j\otimes v_i.
\end{align*}
The resulting element of $M_{\lambda-\wt(v)-\wt(w)}\otimes W\otimes V$ is
\[\sum_i x_{\lambda-\wt(v)-\wt(w)}\otimes a_i(\lambda-\wt(v)) w\otimes v_i + \dots\]
which proves the claim.
\end{proof}

Moreover, in \cite[Section 5]{EtingofVarchenkoExchange} Etingof and Varchenko have introduced an $\h$-bialgebroid $F(G)$.

\begin{thm}
Consider the monoidal functor
\[\Rep(G)\xrightarrow{\free_G} \HC(G)\xrightarrow{\res^{\gen}}\HC(H)^{\gen}.\]
It admits a colimit-preserving right adjoint; denote by $B\in\HC(H)^{\gen}\otimes \HC(H)^{\gen}$ the Harish-Chandra bialgebroid corresponding to this monoidal functor constructed in \cref{HarishChandraTannakaReconstruction}. Then we have an isomorphism of $\h$-bialgebroids
\[B\otimes_{(\U\h)^{\gen}\otimes (\U\h)^{\gen}} (\Frac(\U\h)\otimes \Frac(\U\h))\cong F(G).\]
\label{thm:FGexchange}
\end{thm}
\begin{proof}
By \cref{thm:EVclassicalbasis} the functor $\Rep(G)\rightarrow \HC(H)^{\gen}$ factors as
\[\Rep(G)\longrightarrow \Rep(H)\longrightarrow \HC(H)^{\gen}.\]
Under this composite a finite-dimensional $G$-representation $V\in\Rep(G)$ is sent to a compact projective object $(\U\h)^{\gen}\otimes V\in\HC(H)^{\gen}$, so this functor admits a colimit-preserving right adjoint.

Since $G$ is semisimple, by \cref{thm:Tpresentation} the Harish-Chandra bialgebroid $B$ is isomorphic to
\[\bigoplus_{V\in\Irr(G)} ((\U\h)^{\gen}\otimes V^\vee)\boxtimes ((\U\h)^{\gen}\otimes V)\in\HC(H)^{\gen}\otimes \HC(H)^{\gen},\]
where the sum is over isomorphism classes of irreducible finite-dimensional $G$-representations.
In particular, we get an isomorphism of $(\Frac(\U\h),\Frac(\U\h))$-bimodules
\[B\otimes_{(\U\h)^{\gen}\otimes (\U\h)^{\gen}} (\Frac(\U\h)\otimes \Frac(\U\h))\cong F(G).\]

In the notations of \cref{thm:Tpresentation} and \cite[Section 5]{EtingofVarchenkoExchange}, the isomorphism is given by 
\begin{align*}
T_V & \mapsto L^V, \\
t(f(\lambda)) & \mapsto f(\lambda^1), \\
s(f(\lambda)) & \mapsto f(\lambda^2).
\end{align*}
It is clear that this isomorphism preserves coproduct, counit and unit and the only nontrivial check is that the product is preserved as well. The relations $(18),(19)$ in \textit{loc. cit.} are clearly satisfied. For $(20)$, the claim follows from \cref{prop:classicalfusion}.
\end{proof}

\subsection{Quantum parabolic restriction}
\label{sect:quantumres}

In this section we define parabolic restriction in the setting of quantum groups; we use the notation from \cref{sect:quantumgroups}.

\begin{defn}
The \defterm{universal quantum category $\cO$} is the category $\cO_q^{\univ}$ of $(\Uq(\g),\Uq(\h))$-bimodules whose diagonal $\Uq(\b)$-action is integrable. The \defterm{universal quantum Verma module} is the object
\[M^{\univ} = \Uq(\g)\otimes_{\Uq(\b)}\Uq(\h)\in\cO_q^{\univ}.\]
\end{defn}

\begin{remark}
As in the classical case, we may identify $\cO_q^{\univ}$ with the full subcategory of $\LMod_{\Uq(\g)}(\Rep_q(H))$ of $\Uq(\g)$-modules whose $\Uq(\n)$-action is locally finite.
\end{remark}

We will now define a quantum analog of the bimodules \eqref{eq:classicalcorrespondence}:
\begin{equation}
\HC_q(G)\curvearrowright \cO_q^{\univ} \curvearrowleft \HC_q(H).
\end{equation}

\begin{lm}
Suppose $X\in\HC_q(G)$. The left $\Uq(\g)^{\lf}$-module structure on $X\otimes_{\Uq(\g)^\lf} \Uq(\g)$ has a canonically extension to a $\Uq(\g)$-module structure. Moreover, the left $\Uq(\n)$-action on $X\otimes_{\Uq(\g)^\lf} M^{\univ}$ is locally finite.
\label{lm:Uqlfnaction}
\end{lm}
\begin{proof}
Recall from \cref{rmk:inversefieldgoal} that the left action of $a\in\Uq(\g)^\lf$ on $x\in X$ is
\[a\triangleright x = (\ad a_{(1)})(x) \triangleleft a_{(2)},\]
where $\ad$ refers to the diagonal $\Uq(\g)$-action on $X$. So, we may extend the left $\Uq(\g)^\lf$-action on the relative tensor product $X\otimes_{\Uq(\g)^\lf} \Uq(\g)$ to a $\Uq(\g)$-action by the formula
\[a\triangleright (x\otimes h) = (\ad a_{(1)})(x)\otimes a_{(2)} h\]
for $a\in\Uq(\g)$ an $x\otimes h\in X\otimes_{\Uq(\g)^\lf} \Uq(\g)$. It is well-defined (i.e. descends to the relative tensor product) using the formula $(\ad a_{(1)})(l)a_{(2)} = al$ for any $a\in\Uq(\g)$ and $l\in\Uq(\g)^\lf$.

The diagonal $\Uq(\n)$-action on $X\otimes_{\Uq(\g)^\lf} M^{\univ}$ is locally finite since it is so on $X$ and $M^{\univ}$.
\end{proof}

A $\Uq(\g)^\lf$-bimodule $X\in\HC_q(G)$ acts on a $(\Uq(\g), \Uq(\h))$-bimodule $M\in\cO_q^{\univ}$ via
\[X, M\mapsto X\otimes_{\Uq(\g)^{\lf}} M.\]
By construction it is a $(\Uq(\g), \Uq(\h))$-bimodule. Since the diagonal $\Uq(\n)$-action on $X$ and the left $\Uq(\n)$-action on $M$ are locally finite, so is the left $\Uq(\n)$-action on this bimodule. In particular, it lies in $\cO_q^{\univ}$.

For a $\Uq(\h)$-bimodule $X\in\HC_q(H)$ and a $(\Uq(\g), \Uq(\h))$-bimodule $M\in\cO_q^{\univ}$ the action is
\[M, X\mapsto M\otimes_{\Uq(\h)} X.\]

Let
\[\act_G\colon \HC_q(G)\longrightarrow \cO_q^{\univ},\qquad \act_H\colon \HC_q(H)\longrightarrow \cO_q^{\univ}\]
be the actions of $\HC_q(G)$ and $\HC_q(H)$ on the universal Verma module $M^{\univ}$.

\begin{defn}
The \defterm{parabolic restriction} and \defterm{parabolic induction} are the lax monoidal functors
\begin{align*}
\res&=\act^\R_H\circ \act_G\colon \HC_q(G)\longrightarrow \HC_q(H) \\
\ind&=\act^\R_G\circ \act_H\colon \HC_q(H)\longrightarrow \HC_q(G).
\end{align*}
\end{defn}

We have a functor
\[(-)^{\Uq(\n)}\colon \cO_q^{\univ}\longrightarrow \HC_q(H)\]
of $\Uq(\n)$-invariants.

\begin{prop}
The functor $(-)^{\Uq(\n)}\colon \cO_q^{\univ}\rightarrow \HC_q(H)$ is right adjoint to $\act_H\colon \HC_q(H)\rightarrow \cO_q^{\univ}$.
\end{prop}
\begin{proof}
For $M\in\cO_q^{\univ}$ and $X\in\HC_q(H)$ we have
\begin{align*}
\Hom_{\cO_q^{\univ}}(\act_H(X), M)&= \Hom_{\cO_q^{\univ}}(\Uq(\g)\otimes_{\Uq(\b)} X, M) \\
&\cong \Hom_{\BiMod{\Uq(\b)}{\Uq(\h)}}(X, M) \\
&\cong \Hom_{\HC_q(H)}(X, M^{\Uq(n)}).
\end{align*}
\end{proof}

\begin{prop}
The natural morphism $\Uq(\h)\rightarrow \res(\Uq(\g)^\lf)$ is an isomorphism.
\end{prop}
\begin{proof}
The proof is similar to the proof of \cref{prop:classicalresunit}, where we again use the fact that the quantum Verma module is irreducible for generic parameters \cite[Theorem 4.15]{VoigtYuncken}.
\end{proof}

A weight for a $\Uq(\g)$-module is specified by an element of $H(k)\cong \Hom(\Lambda, k^\times)$. We will use an additive notation for weights, so that a vector $v$ of weight $\lambda$ satisfies $K_\mu v= q^{(\lambda, \mu)} v$. For a root $\alpha$ we denote $q_\alpha = q^{(\alpha, \alpha)/2}$.

\begin{defn}
A weight $\lambda$ is \defterm{generic} if $q^{(\alpha, \lambda)}\not\in \pm q_\alpha^\Z$ for every root $\alpha$. Denote by $H^\gen\subset H$ the subset of generic weights. We denote by $\HC_q^{\gen}(H)\subset \HC_q(H)$ and $\cO_q^{\univ,\gen}\subset \cO_q^{\univ}$ the full subcategories of modules with generic $\Uq(\h)$-weights. Let $\Uq(\h)^{\gen}\subset\Frac(\Uq(\h))$ be the subspace of rational functions on $H$ regular on $H^{\gen}$.
\end{defn}

We denote by
\[M^{\univ, \gen} = \Uq(\g)\otimes_{\Uq(\b)} \Uq(\h)^{\gen}\]
the universal quantum Verma module with generic highest weights.

A generalization of the extremal projector to quantum groups was introduced in \cite{KhoroshkinTolstoy}.

\begin{thm}
There is an extension $T_q(\g)$ of $\Uq(\g)$ obtained by replacing $\Uq(\h)\subset\Uq(\g)$ with $\Frac(\Uq(\h))$ and considering certain power series. There is an element $P\in T_q(\g)$ satisfying the following properties:
\begin{enumerate}
\item $\Uq^{>0}(\n) P = P \Uq^{>0}(\n_-) = 0$.

\item $P-1\in T_q(\g)\Uq^{>0}(\n)\cap \Uq^{>0}(\n_-) T_q(\g)$.
\end{enumerate}
The action of $P$ is well-defined on left $\Uq(\g)$-modules whose $\Uq(\n)$-action is locally nilpotent and which have generic $\Uq(\h)$-weights.
\label{thm:quantumextremalprojector}
\end{thm}

\begin{example}
Consider $\Uq(\sl_2)$ with generators $E,K,F$ as in \cref{ex:Uqsl2lf}. Let $[n] = \frac{q^n-q^{-n}}{q-q^{-1}}$ be the quantum integer, $[n]! = \prod_{j=1}^n [j]!$ the quantum factorial and
\[[h+n] = \frac{Kq^n-K^{-1}q^{-n}}{q-q^{-1}}\in\Uq(\sl_2)\]
for $n\in\Z$. Then the extremal projector is (see e.g. \cite[Section 9]{KhoroshkinOgievetsky})
\[P = \sum_{n=0}^\infty \frac{(-1)^n}{[n]!} g_n^{-1} F^n E^n,\]
where $g_n = \prod_{j=1}^n [h+j+1]$.
\end{example}

Completely analogously to the proof of \cref{thm:Ofreerank1}, one proves the following statement.

\begin{thm}
The category $\cO_q^{\univ, \gen}$ is free of rank $1$ as a $\HC_q(H)^{\gen}$-module category.
\end{thm}

\begin{cor}
The functor $\res^{\gen}\colon \HC_q(G)\rightarrow \HC_q(H)^{\gen}$ is strongly monoidal and colimit-preserving.
\label{cor:quantumresmonoidal}
\end{cor}

Similarly to the classical case, parabolic restriction of a free Harish-Chandra bimodule is free.

\begin{thm}
For every $V\in\Rep_q(G)$ the morphism
\[(V\otimes M^{\univ, \gen})^{\Uq(\n)}\subset V\otimes M^{\univ, \gen} \longrightarrow V\otimes \Uq(\h)^{\gen},\]
where the second morphism is induced by the projection $M^{\univ, \gen}\rightarrow \Uq(\h)^{\gen}$ onto highest weights, defines a natural isomorphism witnessing commutativity of the diagram
\[
\xymatrix{
\Rep_q(G) \ar[d] \ar^{\free_G}[r] & \HC_q(G) \ar^{\res^{\gen}}[d] \\
\Rep_q(H) \ar^{\free_H}[r] & \HC_q(H)^{\gen}
}
\]
\label{thm:EVquantumbasis}
\end{thm}

Combining \cref{cor:quantumresmonoidal} and \cref{thm:EVquantumbasis} we obtain a monoidal structure on the composite
\[\Rep_q(G)\longrightarrow \Rep_q(H)\longrightarrow \HC_q(H)^{\gen}.\]
In particular, by \cref{prop:dynamicaltwistmonoidal} this gives rise to linear maps
\[J_{V, W}(\lambda)\colon V\otimes W\rightarrow V\otimes W,\]
rational functions on $H$.

\begin{example}
Consider $G=\SL_2$ and $V\in\Rep_q(\SL_2)$ the irreducible two-dimensional representation with the basis $\{v_+,v_- \}$ such that
\[K v_+ = q v_+,\qquad K v_- = q^{-1} v_+,\qquad F v_+ = v_-.\]
The isomorphism $\U_q(\h)^{\gen} \otimes V \rightarrow (V \otimes M^{\univ,\gen})^{\Uq(\n)}$ is given by
\begin{align*} 
1 \otimes v_+ & \mapsto v_+ \otimes 1, \\
1 \otimes v_- & \mapsto v_- \otimes 1 - q^{-1}  v_+ \otimes F\cdot [h]^{-1} x^{\univ}. 
\end{align*}
Then the matrix of $J_{V,V}(\lambda)$ in the basis $\{ v_+ \otimes v_+,v_+ \otimes v_-, v_- \otimes v_+, v_- \otimes v_- \}$ is given by
\[ \left( \begin{array}{cccc} 
1 & 0 & 0 & 0 \\ 
0 & 1 & -q^{-\lambda-1}[\lambda+1]^{-1} & 0 \\
0 & 0 & 1 & 0 \\
0 & 0 & 0 & 1
\end{array} \right).\]
\label{ex:SL2J}
\end{example}

Our convention for the coproduct on $\Uq(\g)$ follows \cite[Lemma 3.1.4]{Lusztig}. For two $\Uq(\g)$-modules $V,W$ we denote by $V\botimes W$ the vector space $V\otimes W$ equipped with the $\Uq(\g)$-module structure via the opposite coproduct:
\[h\triangleright (v\otimes w) = h_{(2)}\triangleright v\otimes h_{(1)}\triangleright w.\]

Consider $V,W\in\Rep_q(G)$. Similarly to the classical case, Etingof and Varchenko \cite{EtingofVarchenkoExchange} have introduced the \defterm{fusion matrix} $J^{EV}_{V, W}(\lambda)\colon V\otimes W\rightarrow V\otimes W$, a rational function on $H$, using intertwiners of quantum Verma modules. Note, however, that in our notations they are considering maps
\[\Phi^v_{\lambda}\colon M_\lambda\longrightarrow M_{\lambda-\wt(v)}\botimes V\]
with the property that $\Phi^v_{\lambda}(x_\lambda) = x_\lambda\otimes v + \dots$ Analogously to \cref{prop:classicalfusion} we have the following statement.

\begin{prop}
Let $V,W\in\Rep_q(G)$. The maps $J_{V, W}$ and $J^{EV}_{W, V}$ are related as follows:
\[J_{V, W}(\lambda) = \tau J^{EV}_{W, V} \tau.\]
\label{prop::quantumfusion}
\end{prop}

In \cite[Section 5]{EtingofVarchenkoExchange} Etingof and Varchenko have introduced an $\h$-bialgebroid $F_q(G)$. Analogously to \cref{thm:FGexchange} we obtain the following statement.

\begin{thm}
Consider the monoidal functor
\[\Rep_q(G)\xrightarrow{\free_G} \HC_q(G)\xrightarrow{\res^{\gen}}\HC_q(H)^{\gen}.\]
It admits a colimit-preserving right adjoint; denote by $B\in\HC_q(H)^{\gen}\otimes \HC_q(H)^{\gen}$ the Harish-Chandra bialgebroid corresponding to this monoidal functor constructed in \cref{HarishChandraTannakaReconstruction}. Then we have an isomorphism of $\h$-algebroids
\[B\otimes_{\Uq(\h)^{\gen}\otimes \Uq(\h)^{\gen}}\otimes (\Frac(\Uq(\h))\otimes \Frac(\Uq(\h)))\cong F_q(G).\]
\label{thm:FqGexchange}
\end{thm}

\section{Dynamical Weyl groups}

In this section we introduce a Weyl symmetry of the parabolic restriction functors $\res\colon \HC(G)\rightarrow \HC(H)$ and $\res\colon \HC_q(G)\rightarrow \HC_q(H)$ introduced in \cref{sect:fusion} and relate it to dynamical Weyl groups.

\subsection{Classical Zhelobenko operators}

Fix a group $G$ and its Lie algebra $\g$ as in \cref{sect:classicalres}. Recall that the Weyl group is
\[W = \N(H) / H,\]
where $\N(H)$ is the normalizer of $H$ in $G$. Denote by $\hat{W}$ the corresponding braid group generated by simple reflections $s_\alpha\in W$ with the relation (for $\alpha\neq \beta$)
\[\underbrace{s_\alpha s_\beta s_\alpha \dots}_{m_{\alpha\beta}} = \underbrace{s_\beta s_\alpha s_\beta \dots}_{m_{\alpha\beta}},\]
where $m_{\alpha\beta}$ is the Coxeter matrix. There is a canonical map $W\rightarrow \hat{W}$ which sends a reduced expression $w=s_1\cdot\dots\cdot s_n\in W$ to the corresponding element in $\hat{W}$.

We may lift $w\in W$ to elements $T_w\in \N(H)$ satisfying the braid relations. Moreover, for a simple reflection $s_\alpha\in W$ the element $T_{s_\alpha}^2\in H$ has order at most 2 \cite{Tits}. For concreteness, we assume that the elements $T_w$ act via the $q=1$ version of Lusztig's operators $T'_{w, 1}$ as in \cite[Section 5.2.1]{Lusztig}.

Denote by $\rho\in\h^*$ the half-sum of positive roots. For an element $w\in W$ and $\lambda\in\h^*$ denote by $w\cdot \lambda\in\h^*$ the \defterm{dot action}:
\[w\cdot\lambda = w(\lambda+\rho) - \rho.\]
The induced action on $h\in\h\subset\U\h$ is denoted by
\[w\cdot h = w(h) + \langle h, w(\rho) - \rho\rangle,\]
where the usual $W$ action on $\U\h$ is simply denoted by $w(d)$ for $d\in\U\h$.

Recall that for a right $\U\g$-module $X$, $X\otimes_{\U\g} M^{\univ}\cong X/X\n$. In the study of Mickelsson algebras Zhelobenko \cite{ZhelobenkoCocycles} has introduced a collection of operators acting on $\U\g$-bimodules for each element of the Weyl group. We refer to \cite[Section 6]{KhoroshkinOgievetsky} for the proof of the following results.

\begin{thm}
Suppose $X\in\HC(G)$. Suppose $\alpha$ is a simple root and denote by $\{e_\alpha, h_\alpha, f_\alpha\}$ the standard generators of the corresponding $\sl_2$ subalgebra $\g_\alpha\subset \g$. Consider the \defterm{Zhelobenko operator} $\bq_\alpha\colon X\rightarrow X$ given by an infinite series
\[\bq_\alpha(x) = \sum_{n=0}^\infty \frac{(-1)^n}{n!} (\ad e_\alpha)^n(T_{s_\alpha}(x)) f_\alpha^n g_{n, \alpha}^{-1},\]
where
\[g_{n,\alpha} = \prod_{j=1}^n (h_\alpha-j+1)\]
and $\ad e_\alpha$ refers to the diagonal $\g$-action on $X$. Then the operators $\bq_\alpha$ descend to well-defined linear isomorphisms
\[\bq_\alpha\colon (X\otimes_{\U\g} M^{\univ, \gen})^N\longrightarrow (X\otimes_{\U\g} M^{\univ, \gen})^N\]
which satisfy the following relations:
\begin{enumerate}
\item $\bq_\alpha((\ad d)(x)) = (\ad s_\alpha(d))(\bq_\alpha(x))$ for every $d\in\h$ and $x\in X$.

\item $\bq_\alpha(dx) = (s_\alpha\cdot d)\bq_\alpha(x)$ for every $d\in\h$ and $x\in X$.

\item $\underbrace{\bq_\alpha\bq_\beta\bq_\alpha\dots}_{m_{\alpha\beta}} = \underbrace{\bq_\beta\bq_\alpha\bq_\beta\dots}_{m_{\alpha\beta}}$ for $\alpha\neq \beta$.

\item $\bq_\alpha^2(x) = (h_\alpha+1)^{-1}T_{s_\alpha}^2(x) (h_{\alpha}+1)$
 for every $x\in X$.
 \label{eq::zhelobenko_squareprop}
\end{enumerate}
\label{thm:classicalZhelobenko}
\end{thm}

For an element $w\in W$ with a reduced decomposition $w=s_{\alpha_1}\dots s_{\alpha_n}$ we define
\[\bq_w = \bq_{\alpha_{i_1}}\dots \bq_{\alpha_{i_n}}.\]
The third relation in \cref{thm:classicalZhelobenko} shows that $\bq_w$ is independent of the chosen decomposition.

In addition, we have the following important multiplicativity property of the Zhelobenko operators proven in \cite[Theorem 3]{KhoroshkinOgievetsky}.

\begin{thm}
Let $X,Y\in\HC(G)$ and take $x\in X$ and $y\in Y$, where $\n y\in Y\n$. Then we have an equality
\[\bq_w(x\otimes y) = \bq_w(x)\otimes \bq_w(y)\]
in $X\otimes_{\U\g} Y\otimes_{\U\g} M^{\univ, \gen}$.
\label{thm:classicalZhelobenkoMultiplicative}
\end{thm}

\subsection{Classical dynamical Weyl group}
\label{sect:classicaldynamicalWeyl}

Given a group $G$ we may regard it as a discrete monoidal category $\Cat(G)$. Let us recall the notion of a $G$-action on a monoidal category and the category of $G$-equivariant objects (see e.g. \cite[Section 2.7]{EGNO}).

\begin{defn}
Let $\cC\in\PrL$ be a monoidal category. A \defterm{$G$-action on $\cC$} is a monoidal functor
\[\Cat(G)\longrightarrow \Fun^{\L, \otimes}(\cC, \cC)\]
to the monoidal category of monoidal colimit-preserving endofunctors on $\cC$.
\end{defn}

Explicitly, for every element $g\in G$ we have a monoidal functor $S_g\colon \cC\rightarrow \cC$ together with a natural isomorphism $S_e\cong \id$ and natural isomorphisms $S_{gh}\cong S_g\circ S_h$ for a pair of elements $g,h\in G$ satisfying an associativity axiom.

\begin{defn}
Suppose $\cC$ is a monoidal category with a $G$-action. A \defterm{$G$-equivariant object} is an object $x\in \cC$ equipped with isomorphisms $S_g(x)\cong x$ compatible with the isomorphisms $S_{gh}\cong S_g\circ S_h$ and $S_e\cong \id$. We denote by $\cC^G$ the category of $G$-equivariant objects.
\end{defn}

The category $\HC(H)\cong \LMod_{\U\h}(\Rep H)$ carries a natural action of the Weyl group $W$ defined as follows. Let us regard $X\in\HC(H)$ as a $\U\h$-bimodule. Then the action of $w\in W$ twists the left and the right $\U\h$-action by the dot action: $S_w(X) = X$ as a plain vector space with the $\U\h$-bimodule structure given by
\[d\triangleright^w x = (w\cdot d)\triangleright x,\qquad x\triangleleft^w d = x\triangleleft(w\cdot d)\]
for $x\in X$ and $d\in\U\h$. The dot action of $W$ on $\h$ is given by affine transformations, so the corresponding diagonal $\h$-action on $S_w(X)$ is given by its linear part, i.e. we twist the diagonal $\h$-action on $X$ by the usual $W$-action. By construction $S_e=\id$ and $S_{w_1w_2} = S_{w_1}\circ S_{w_2}$. Moreover, the identity map of vector spaces
\[S_w(X)\otimes_{\U\h} S_w(Y)\longrightarrow S_w(X\otimes_{\U\h} Y)\]
together with the dot action
\[\U\h\longrightarrow S_w(\U\h)\]
define a monoidal structure on the collection $\{S_w\}_{w\in W}$.

The functor
\[\free\colon \Rep(H)\longrightarrow \HC(H)\]
is naturally $W$-equivariant, where the maps
\begin{equation}
\U\h\otimes S_w(V)\rightarrow S_w(\U\h\otimes V)
\label{eq:freeWequivariance}
\end{equation}
are given by the dot action on the $\U\h$ factor.

Restricting the $W$-action on $\HC(H)$ under the quotient map $\hat{W}\rightarrow W$ from the braid group we obtain a natural action of $\hat{W}$ on $\HC(H)$.

Recall that by \cref{cor:resmonoidal} the parabolic restriction functor
\[\res^{\gen}\colon \HC(G)\longrightarrow \HC(H)^{\gen}\]
given by $X\mapsto (X\otimes_{\U\g} M^{\univ})^N$ is monoidal. We will now show that it factors through $\hat{W}$-invariants.

\begin{thm}
The Zhelobenko operators define a factorization
\[
\xymatrix{
\HC(G) \ar@{-->}[r] \ar_{\res^{\gen}}[dr] & \HC(H)^{\gen, \hat{W}} \ar[d] \\
& \HC(H)^{\gen}
}
\]
of $\res^{\gen}\colon \HC(G)\rightarrow \HC(H)^{\gen}$ through a monoidal functor $\res^{\gen}\colon\HC(G)\rightarrow \HC(H)^{\gen, \hat{W}}$.
\label{thm:classicalZhelobenkoFactorization}
\end{thm}
\begin{proof}
Let us first construct a factorization of $\res^{\gen}$ through $\HC(H)^{\gen, \hat{W}}\rightarrow \HC(H)^{\gen}$ as a plain (non-monoidal) functor. Since the braid group $\hat{W}$ is generated by simple reflections $\{s_\alpha\}$, for $X\in\HC(G)$ we have to specify natural isomorphisms
\[\res^{\gen}(X)\xrightarrow{\sim} S_{s_\alpha}(\res^{\gen}(X))\]
satisfying the braid relations. We define them to be the Zhelobenko operators $\bq_\alpha$. The compatibility with the $\U\h$-bimodule action follows from parts (1) and (2) of \cref{thm:classicalZhelobenko}. The braid relations follow from part (3) of the same theorem.

Next, we have to construct a monoidal structure on $\HC(G)\rightarrow \HC(H)^{\gen, \hat{W}}$ compatible with the one on $\res^{\gen}\colon \HC(G)\rightarrow \HC(H)^{\gen}$ which we recall is given by \eqref{eq:resmonoidal}. The unit map is the natural inclusion $\U\h\hookrightarrow (M^{\univ})^N$.

We begin by showing compatibility with the tensor products. By \cref{prop:coinvinv} the functor of $N$-invariants $\cO^{\univ}\rightarrow \HC(H)^{\gen}$ is exact. In particular, we may exchange the order of left $N$-invariants and right $\n$-coinvariants in the definition of $\res(X) = (X/X\n)^N$. But then the diagram
\begin{equation}
\xymatrix{
\res^{\gen}(X)\otimes_{(\U\h)^{\gen}} \res^{\gen}(Y) \ar[r] \ar^{\bq_\alpha\otimes \bq_\alpha}[d] & \res^{\gen}(X\otimes_{\U\g} Y) \ar^{\bq_\alpha}[d] \\
S_{s_\alpha}(\res^{\gen}(X))\otimes_{(\U\h)^{\gen}} S_{s_\alpha}(\res^{\gen}(Y)) \ar[r] & S_{s_\alpha}(\res^{\gen}(X\otimes_{\U\g} Y))
}
\label{diag::zhelobenko_monoidal}
\end{equation}
is commutative by \cref{thm:classicalZhelobenkoMultiplicative}.

Next, we have to show compatibility with the unit maps. Consider the diagram
\[
\xymatrix{
\U\h \ar[r] \ar[d] & (M^{\univ})^N \ar^{\bq_\alpha}[d] \\
S_{s_\alpha}(\U\h) \ar[r] & S_{s_\alpha}((M^{\univ})^N)
}
\]
To show that it is commutative, we have to compute the action of $\bq_{\alpha}$ on $\U\h\hookrightarrow M^{\univ}$. By part (2) of \cref{thm:classicalZhelobenko} $\bq_{\alpha}(d\cdot 1) = (s_\alpha\cdot d) \bq_{\alpha}(1)$, where $d\in \U\h$ and $1\in\U\g$ is the unit. But it is obvious from the explicit formula for $\bq_{\alpha}$ that $\bq_{\alpha}(1)=1$.
\end{proof}

Let us now analyze the composite monoidal functor
\[\Rep(G)\xrightarrow{\free_G}\HC(G)\xrightarrow{\res^{\gen}} \HC(H)^{\gen, \hat{W}}.\]
Recall that by \cref{thm:EVclassicalbasis} we have a commutative diagram
\[
\xymatrix{
\Rep(G) \ar^{\free_G}[r] \ar[d] & \HC(G) \ar^{\res^{\gen}}[d] \\
\Rep(H) \ar^{\free_H}[r] & \HC(H)^{\gen}
}
\]
of plain (non-monoidal) categories.

Consider $V\in\Rep(G)$. Using the natural isomorphism
\[\res^{\gen}(\U\g\otimes V)\cong (\U\h)^{\gen}\otimes V\]
in $\HC(H)^{\gen}$ provided by the above diagram we obtain that the $\hat{W}$-invariance of $\res^{\gen}(\U\g\otimes V)$ boils down to maps $(\U\h)^{\gen}\otimes V\rightarrow (\U\h)^{\gen}\otimes S_w(V)$ obtained via the composite
\[
(\U\h)^{\gen}\otimes V \xleftarrow{\sim} \res^{\gen}(\U\g\otimes V)\xrightarrow{\bq_w} S_w(\res^{\gen}(\U\g\otimes V))\xrightarrow{\sim} S_w((\U\h)^{\gen}\otimes V)\xrightarrow{\sim} (\U\h)^{\gen}\otimes S_w(V).
\]

Such maps are uniquely determined by their value on $1\otimes v$, which gives linear maps
\[A_{w, V}(\lambda)\colon V\longrightarrow V\]
depending rationally on a parameter $\lambda\in\h^*$.

Let $V,U\in\Rep(G)$ and recall the matrix $J_{V, U}(\lambda)\colon V\otimes U\rightarrow V\otimes U$ defined in \cref{sect:classicalres} which controls the monoidal structure on the composite $\Rep(G)\rightarrow\Rep(H)\xrightarrow{\free_H} \HC(H)^{\gen}$.

\begin{prop}
For any simple reflection $s_\alpha$ and $V,U\in\Rep(G)$ we have an equality
\[A_{s_\alpha, V\otimes U}(\lambda) J_{V,U}(\lambda) = J_{V,U}(s_\alpha\cdot \lambda) A_{s_\alpha, U}^{(1)}(\lambda)A_{s_\alpha, U}^{(2)}(\lambda-h^{(1)})\]
of rational functions $\h^*\rightarrow \End(V\otimes U)$, where $A^{(1)}$ denotes $A\otimes 1$ and $A^{(2)}$ denotes $1\otimes A$.
\label{prop:classicaldynamicalWeylgroupmultiplicative}
\end{prop}
\begin{proof}
Consider the diagram
\[
\xymatrix{
((\U\h)^{\gen}\otimes V)\otimes_{(\U\h)^{\gen}} ((\U\h)^{\gen}\otimes U) \ar[r] \ar[d] & (\U\h)^{\gen}\otimes V\otimes U\ar[d] \\
\res^{\gen}(\U\g\otimes V)\otimes_{(\U\h)^{\gen}} \res^{\gen}(\U\g\otimes U) \ar[r] \ar^{\bq_\alpha\otimes \bq_\alpha}[d] & \res^{\gen}((\U\g\otimes V)\otimes_{\U\g} (\U\g\otimes U)) \ar^{\bq_\alpha}[d] \\
S_{s_\alpha}(\res^{\gen}(\U\g\otimes V))\otimes_{(\U\h)^{\gen}} S_{s_\alpha}(\res^{\gen}(\U\g\otimes U)) \ar[r] \ar[d] & S_{s_\alpha}(\res^{\gen}((\U\g\otimes V)\otimes_{\U\g} (\U\g\otimes U)))  \ar[d] \\
(\U\h)^{\gen}\otimes V\otimes U \ar[r] & (\U\h)^{\gen}\otimes V\otimes U
}
\]
where the middle square is \eqref{diag::zhelobenko_monoidal}.

The left vertical arrow is $A_{s_\alpha, V}^{(1)}(\lambda)A_{s_\alpha, U}^{(2)}(\lambda-h^{(1)}) $ and the right vertical arrow is $A_{s_\alpha, V\otimes U}(\lambda)$. Using the isomorphism \eqref{eq:freeWequivariance} the bottom horizontal arrow is $J_{V,U}(s_{\alpha}\cdot \lambda)$.
\end{proof}

Let us now compute a particular example of the operators $A_{w, V}(\lambda)$. Consider $G=\SL_2$, $V$ the two-dimensional irreducible representation, $H\subset G$ the subgroup of diagonal matrices and $w$ the unique simple reflection. We can lift it to the matrix $T\in\N(H)$ given by
\[
T = \left(\begin{array}{cc} 0 & -1 \\ 1 & 0\end{array}\right).
\]
Let $\{e, h, f\}$ be the standard basis of $\sl_2$. Let $\{v_+, v_-\}$ be the basis of $V$ such that
\[hv_+ = v_+, \qquad hv_- = -v_-,\qquad fv_+ = v_-.\]

\begin{prop}
The action of $A_{w, V}(\lambda)$ is given as follows:
\begin{align*}
A_{w, V}(\lambda) v_+ &= v_- \\
A_{w, V}(\lambda) v_- &= -\frac{\lambda+2}{\lambda+1} v_+
\end{align*}
\label{prop:classicaldynamicalWeylsl2}
\end{prop}
\begin{proof}
The isomorphism $(\U\h)^{\gen} \otimes V \rightarrow (V \otimes M^{\univ,\gen})^N$ is given by
\begin{align*}
1 \otimes v_+ & \mapsto v_+ \otimes x^{\univ}, \\
1 \otimes v_- & \mapsto v_- \otimes x^{\univ} - v_+ \otimes fh^{-1} x^{\univ},
\end{align*}
where $x^{\univ}\in M^{\univ,\gen}$ is the generator. We have 
\[ \bq_w(v_+ \otimes 1)  = \sum\limits_n \frac{(-1)^n}{n!} \ad_e^n(v_- \otimes 1) f^n g_n^{-1} = v_- \otimes 1 - v_+ \otimes fh^{-1}, \]
hence $A_{w,V}(\lambda)(v_+) = v_-$. To compute $A_{w,V}(\lambda)(v_-)$, we use property \eqref{eq::zhelobenko_squareprop} from \cref{thm:classicalZhelobenko}, namely, 
\[ \bq_w(v_- \otimes 1 - v_+ \otimes fh^{-1}) = \bq_w^2 (v_+ \otimes 1) = - (h+1)^{-1}(v_+ \otimes 1)(h+1) = -h(h+1)^{-1}(v_+ \otimes 1). \]
Under identification $S_w(\res^{\gen} (V \otimes \U\g)) \cong (\U\h)^{\gen} \otimes S_w(V)$, we have
\[ \bq_w(v_- \otimes 1 - v_+ \otimes fh^{-1}) \mapsto -\frac{w\cdot \lambda }{w\cdot\lambda + 1} \otimes v_+ = - \frac{\lambda+2}{\lambda+1} \otimes v_+,  \]
and the claim follows.

\end{proof}

We return to the case of arbitrary $G$. Recall that Tarasov and Varchenko \cite{TarasovVarchenko} have introduced the dynamical Weyl group, i.e. a collection of operators $A^{TV}_{w, V}(\lambda)\colon V\rightarrow V$ for every finite-dimensional $\g$-representation $V$ and $w\in W$ depending rationally on the parameter $\lambda\in\h^*$. We will now prove that the operators $A_{w, V}$ constructed from the Zhelobenko operators coincide with the dynamical Weyl group.

\begin{thm}
For any $V\in\Rep(G)$ and $w\in W$ we have an equality of rational functions
\[A^{TV}_{w, V}(\lambda) = A_{w, V}(\lambda).\]
\label{thm:TVdynamicalWeylGroup}
\end{thm}
\begin{proof}
Both $A^{TV}_{w, V}(\lambda)$ and $A_{w, V}(\lambda)$ are given by products in terms of simple reflections, so it is enough to establish the fact for a simple reflection $w=s_\alpha$ along a simple root $\alpha$.

In turn, both $A^{TV}_{s_\alpha, V}(\lambda)$ and $A_{s_\alpha, V}(\lambda)$ are defined by considering the corresponding $\sl_2$-subalgebra $\g_\alpha\subset \g$ generated by $\{e_\alpha, h_\alpha, f_\alpha\}$. So, it is enough to prove the claim for $G=\SL_2$.

For a tensor product of representations $A_{w, V}(\lambda)$ satisfies a multiplicativity property given by \cref{prop:classicaldynamicalWeylgroupmultiplicative} and so does $A^{TV}_{w, V}(\lambda)$ (see \cite[Lemma 7]{TarasovVarchenko}, where the relationship between $J_{V, U}(\lambda)$ and $J^{EV}_{V, U}(\lambda)$ is given by \cref{prop:classicalfusion}). Therefore, it is enough to check the equality on the 2-dimensional irreducible representation of $\sl_2$, which follows by comparing the expressions given in \cref{prop:classicaldynamicalWeylsl2} with the explicit expressions given in \cite[Section 2.5]{TarasovVarchenko} (see also \cite[Lemma 5]{EtingofVarchenkoDynamicalWeyl} for an explicit description of the dynamical Weyl group in the 2-dimensional representation of quantum $\sl_2$).
\end{proof}

\subsection{Quantum Zhelobenko operators}

We continue to use notations for quantum groups from \cref{sect:quantumres}. It was shown by Lusztig \cite{Lusztig}, Soibelman \cite{Soibelman} and Kirillov--Reshetikhin \cite{KirillovReshetikhin} that one can introduce an action of the braid group $\hat{W}$ on modules in $\Rep_q(G)$. For $V\in\Rep_q(G)$ and $w\in W$ we denote by $T_w\colon V\rightarrow V$ the corresponding operator of the quantum Weyl group (for definitiveness, we consider $T'_{w, +1}$ in the notation of \cite[Chapter 5]{Lusztig}).

\begin{example}
Consider $\Uq(\sl_2)$ with generators $E,K,F$ as in \cref{ex:Uqsl2lf}, $V\in\Rep_q(\SL_2)$ and $v\in V$ a vector of weight $n$. Then
\[T_w(v) = \sum_{a,b,c;a-b+c=n} (-1)^b q^{-ac+b} \frac{F^a E^b F^c}{[a]![b]![c]!} v\]
for the unique nontrivial element $w\in W$.
\end{example}

The Weyl group $W$ acts in the standard way on the weight lattice $\Lambda$. We introduce the dot action of $W$ on $\Uq(\h)=k[\Lambda]$ by
\[w\cdot K_\mu = K_{w(\mu)} q^{(\mu, w(\rho)-\rho)}\]
for every $\mu\in\Lambda$.

Recall that for a root $\alpha$ we denote $q_\alpha = q^{(\alpha, \alpha)/2}$. The quantum integer is
\[[n]_\alpha = \frac{q_\alpha^n - q_\alpha^{-n}}{q_\alpha-q_\alpha^{-1}}\]
and the quantum factorial is defined similarly. The quantum Zhelobenko operators were introduced in \cite[Section 9]{KhoroshkinOgievetsky}. For the following statement recall \cref{lm:Uqlfnaction} which explains that the infinite sums in the quantum Zhelobenko operators are well-defined.

\begin{thm}
Suppose $X\in\HC_q(G)$. For a simple root $\alpha$ we denote by $\{E_\alpha, K_\alpha, F_\alpha\}$ the corresponding subset of generators of $\Uq(\g)$. Consider the \defterm{quantum Zhelobenko operator} on $X$ given by
\[\bq_\alpha(x) = \sum_{n=0}^\infty \frac{(-1)^n}{[n]_\alpha!} (\ad (K_\alpha^{-1}E_\alpha))^n((\ad T_{s_\alpha})(x)) F_\alpha^n g_{n,\alpha}^{-1},\]
where
\[g_{n,\alpha} = \prod_{j=1}^n [h_\alpha-j+1]_\alpha\]
and $\ad (K_{\alpha}^{-1} E_\alpha)$ refers to the diagonal $\Uq(\g)$-action. Then the operators $\bq_\alpha$ descend to linear isomorphisms
\[(X\otimes_{\Uq(\g)^\lf} M_q^{\univ, \gen})^{\Uq(\n)}\longrightarrow (X\otimes_{\Uq(\g)^\lf} M_q^{\univ, \gen})^{\Uq(\n)}\]
which satisfy the following relations:
\begin{enumerate}
\item $\bq_\alpha((\ad d)(x)) = (\ad s_\alpha (d))(\bq_\alpha(x))$ for every $d\in\Uq(\h)$ and $x\in X$.

\item $\bq_\alpha(dx) = (s_\alpha\cdot d)\bq_\alpha(x)$ for every $d\in\Uq(\h)$ and $x\in X$.

\item $\underbrace{\bq_\alpha\bq_\beta\bq_\alpha\dots}_{m_{\alpha\beta}} = \underbrace{\bq_\beta\bq_\alpha\bq_\beta\dots}_{m_{\alpha\beta}}$ for $\alpha\neq\beta$.
\end{enumerate}
\end{thm}

The third property allows us to define $\bq_w$ for any element $w\in\hat{W}$. We also have a multiplicativity property.

\begin{thm}
Let $X,Y\in\HC_q(G)$ and take $x\in X$ and $y\in Y$, where $\Uq^{>0}(\n)y\in Y\Uq^{>0}(\n)$. Then we have an equality
\[\bq_w(x\otimes y) = \bq_w(x)\otimes \bq_w(y)\]
in $X\otimes_{\Uq(\g)^\lf} Y\otimes_{\Uq(\g)^\lf} M_q^{\univ, \gen}$.
\end{thm}

\subsection{Quantum dynamical Weyl group}

As in \cref{sect:classicaldynamicalWeyl}, quantum Zhelobenko operators define Weyl symmetry of the parabolic restriction functor $\res^{\gen}\colon \HC_q(G)\rightarrow \HC_q(H)^{\gen}$.

The $W$-action on $\HC_q(H)$ is defined similarly to the $W$-action on $\HC(H)$. An element $w\in W$ gives rise to a functor $S_w\colon \HC_q(H)\rightarrow \HC_q(H)$ given as follows. For $X\in\HC_q(H)$ we set $S_w(X) = X$ as a vector space with the $\Uq(\h)$-bimodule structure given by
\[d\triangleright^w x = (w\cdot d)\triangleright x,\qquad x\triangleleft^w d  = x\triangleleft(w\cdot d),\]
where $d\in\Uq(\h)$ and $x\in X$. The functors $\{S_w\}$ have obvious monoidal structures.

Consider the action of the quantum Zhelobenko operators
\[\bq_{\alpha}\colon \res^{\gen}(X)\xrightarrow{\sim} S_{s_\alpha}(\res^{\gen}(X)).\]

\begin{thm}
The quantum Zhelobenko operators define a factorization
\[
\xymatrix{
\HC_q(G) \ar@{-->}[r] \ar_{\res^{\gen}}[dr] & \HC_q(H)^{\gen, \hat{W}} \ar[d] \\
& \HC_q(H)^{\gen}
}
\]
of $\res^{\gen}_q\colon \HC_q(G)\rightarrow \HC_q(H)^{\gen}$ through a monoidal functor $\res^{\gen}\colon\HC_q(G)\rightarrow \HC_q(H)^{\gen, \hat{W}}$.
\end{thm}

By \cref{thm:EVquantumbasis} we have a commutative diagram
\[
\xymatrix{
\Rep_q(G) \ar[d] \ar^{\free_G}[r] & \HC_q(G) \ar^{\res^{\gen}}[d] \\
\Rep_q(H) \ar^{\free_H}[r] & \HC_q(H)^{\gen}
}
\]
which gives rise to a monoidal structure on the composite
\[\Rep_q(G)\longrightarrow \Rep_q(H)\xrightarrow{\free_H} \HC_q(H)^{\gen, \hat{W}}.\]
As in \cref{sect:classicaldynamicalWeyl}, we obtain linear maps $A_{w, V}(\lambda)\colon V\rightarrow V$ for every $V\in\Rep_q(G)$, which are rational functions on $H$. For $V,U\in\Rep_q(G)$ recall the matrix $J_{V, U}(\lambda)\colon V\otimes U\rightarrow V\otimes U$ defined in \cref{sect:quantumres}.

\begin{prop}
For any simple reflection $s_\alpha$ and $V,U\in\Rep_q(G)$ we have an equality
\[A_{s_\alpha, V\otimes U}(\lambda) J_{V,U}(\lambda) = J_{V,U}(s_\alpha\cdot \lambda) A_{s_\alpha, V}^{(1)}(\lambda)A_{s_\alpha, U}^{(2)}(\lambda-h^{(1)})\]
of rational functions $H\rightarrow \End(V\otimes U)$.
\label{prop:quantumdynamicalWeylgroupmultiplicative}
\end{prop}

Let us now compute the operators $A_{w, V}$ for $G=\SL_2$. Consider the irreducible two-dimensional representation $V\in\Rep_q(G)$ with the basis $\{v_+, v_-\}$, such that
\[K v_+ = q v_+,\qquad K v_- = q^{-1} v_+,\qquad F v_+ = v_-.\]

\begin{prop}
The action of $A_{w, V}(\lambda)$ is given as follows:
\begin{align*}
A_{w, V}(\lambda) v_+ &= v_- \\
A_{w, V}(\lambda) v_- &= -\frac{[\lambda+2]}{[\lambda+1]} v_+
\end{align*}
\label{prop:quantumdynamicalWeylsl2}
\end{prop}
\begin{proof}
The isomorphism $\Uq(\h)^{\gen} \otimes V \rightarrow (V \otimes M^{\univ,\gen})^N$ is given by
\begin{align*}
1 \otimes v_+ & \mapsto v_+ \otimes x^{\univ}, \\
1 \otimes v_- & \mapsto v_- \otimes 1 - q^{-1} v_+ \otimes F\cdot [h]^{-1} \cdot x^{\univ}.
\end{align*}

By \cite[Proposition 5.2.2]{Lusztig} we have
\[T_w(v_+)  = v_-,\qquad T_w(v_-) = -qv_+.\]
Therefore,
\[\bq_w(v_+\otimes 1) = \sum_{n=0}^\infty \frac{(-1)^n}{[n]!} (\ad (K^{-1}E))^n(v_-\otimes 1) F^n g_{n}^{-1} = v_-\otimes 1 - q^{-1} v_+\otimes F[h]^{-1},\]
which implies that
\[A_{w, V}(\lambda)v_+ = v_-.\]

Using the formula for the square of the quantum Zhelobenko operator \cite[Corollary 9.6]{KhoroshkinOgievetsky} we obtain
\[\bq_w(v_-\otimes 1 - q^{-1} v_+\otimes F[h]^{-1}) = - [h+1]^{-1}(v_+\otimes 1)[h+1] = -\frac{[h]}{[h+1]} (v_+\otimes 1),\]
which implies that
\[A_{w, V}(\lambda) v_- = -\frac{[\lambda+2]}{[\lambda+1]} v_+.\]
\end{proof}

\begin{remark}
The formulas (9.10) and (9.11) in \cite{KhoroshkinOgievetsky} are missing a sign, see \cite[Proposition 5.2.2]{Lusztig}.
\end{remark}

Etingof and Varchenko \cite{EtingofVarchenkoDynamicalWeyl} have introduced a quantum analog of the dynamical Weyl group, i.e. a collection of rational functions $A^{EV}_{w, V}(\lambda)\colon V\rightarrow V$ for every $V\in\Rep_q(G)$ and $w\in W$. We are now ready to relate $A_{w, V}$ and $A^{EV}_{w, V}$.

\begin{thm}
For any $V\in\Rep_q(G)$ an $w\in W$ we have an equality of rational functions
\begin{equation}
A^{EV}_{w, V}(\lambda) = q^{(w(\rho)-\rho, h)} A_{w, V}(\lambda).
\label{eq:EVdynamicalWeylGroup}
\end{equation}
\end{thm}
\begin{proof}
The proof is analogous to the proof of \cref{thm:TVdynamicalWeylGroup}. Both $A^{EV}_{w, V}$ and $A_{w, V}$ are given by a product over simple reflections, so it is enough to establish this equality for a simple reflection $w=s_\alpha$.

We have $s_\alpha(\rho) = \rho - \alpha$, so
\[q^{(s_\alpha(\rho)-\rho, h)} A_{s_\alpha, V}(\lambda) = q^{-(\alpha, h)} A_{s_\alpha, V}(\lambda) = K_\alpha^{-1} A_{s_\alpha, V}(\lambda).\]
In particular, both sides of the equality \eqref{eq:EVdynamicalWeylGroup} are defined in terms of the corresponding $\Uq(\sl_2)$-subalgebra, so it is enough to restrict our attention to $G=\SL_2$. Using the multiplicativity property of $A^{EV}_{w, V}$ and $A_{w, V}$ given by \cite[Lemma 4]{EtingofVarchenkoDynamicalWeyl} and \cref{prop:quantumdynamicalWeylgroupmultiplicative}, we reduce to the case of the defining representation. The equality on the defining representation of $\SL_2$ follows from comparing the formulas in \cite[Lemma 5]{EtingofVarchenkoDynamicalWeyl} and \cref{prop:quantumdynamicalWeylsl2}.
\end{proof}

\printbibliography

\end{document}